\RequirePackage{tikz-cd}

\documentclass[sn-mathphys]{sn-jnl}

\jyear{2021}
\usepackage[section]{placeins}
\usepackage[capitalize, noabbrev]{cleveref}
  \crefformat{equation}{\textup{#2(#1)#3}}
\usepackage{tabularx}
\usepackage{natbib}

\theoremstyle{thmstyleone}
\newtheorem{theorem}{Theorem}
\newtheorem{proposition}[theorem]{Proposition}
\newtheorem{lemma}[theorem]{Lemma}
\newtheorem{corollary}[theorem]{Corollary}

\theoremstyle{thmstyletwo}

\theoremstyle{thmstylethree}
\newtheorem{definition}{Definition}

\newcommand{\abs}[1]{\lvert {#1} \rvert}

\usepackage{todonotes}

\begin{document}

\title[Robust Detection of Small Holes by the RDAD Filtration]{Detection of Small Holes by the Scale-Invariant Robust Density-Aware Distance (RDAD) Filtration}

\author*[1]{\fnm{Chunyin (Alex)} \sur{Siu}}\email{cs2323@cornell.edu}

\author[2]{\fnm{Gennady} \sur{Samorodnitsky}}\email{gs18@cornell.edu}

\author[2]{\fnm{Christina} \sur{Lee Yu}}\email{cleeyu@cornell.edu}

\author[3]{\fnm{Andrey} \sur{Yao}}\email{awy32@cornell.edu}

\affil*[1]{\orgdiv{Center of Applied Mathematics}, \orgname{Cornell University}, \orgaddress{\street{Frank H.T. Rhodes Hall, Cornell University}, \city{Ithaca}, \postcode{14853}, \state{NY}, \country{USA}}}

\affil[2]{\orgdiv{School of Operations Research and Information Engineering}, \orgname{Cornell University}, \orgaddress{\street{Frank H.T. Rhodes Hall, Cornell University}, \city{Ithaca}, \postcode{14853}, \state{NY}, \country{USA}}}

\affil[3]{\orgdiv{Department of Mathematics}, \orgname{Cornell University}, \orgaddress{\street{Malott Hall, Cornell University}, \city{Ithaca}, \postcode{14853}, \state{NY}, \country{USA}}}

\abstract{
A novel topological-data-analytical (TDA) method is proposed to distinguish, from noise, small holes surrounded by high-density regions of a probability density function. The proposed method is robust against additive noise and outliers. 
Traditional TDA tools, like those based on the distance filtration, often struggle to distinguish small features from noise, because both have short persistences. An alternative filtration, called the Robust Density-Aware Distance (RDAD) filtration, is proposed to prolong the persistences of small holes of high-density regions. This is achieved by weighting the distance function by the density in the sense of Bell et al. The concept of distance-to-measure is incorporated to enhance stability and mitigate noise. The persistence-prolonging property and robustness of the proposed filtration are rigorously established, and numerical experiments are presented to demonstrate the proposed filtration's utility in identifying small holes.
}


\maketitle

\section{Introduction}

Topological data analysis is a non-parametric approach to data analysis that looks for topological features, like connected components, loops and cavities. For some datasets, such features are indeed the dominant features. Consider, for instance, the cosmologically motivated dataset \cite{icke91_cosmic_void_voronoi} on the left of \cref{fig:TDA_illustration_small_features}, which we will revisit in \cref{sec:voronoi}. The regions that the data points avoid form conspicuous holes, and the topological description of the these holes may offer a unique insight into the dataset. Indeed, since the seminal work of \cite{carlsson09_topodata}, TDA has been used for a wide range of applications \cite{chazal21_TDA_survey_data, perea19_TDA_survey_time_series, aktas19_TDA_survey_network, buchet18_TDA_survey_material_science, xu19_cosmic_void_TDA, salch21_TDA_survey_MRI}.

Traditionally, large topological features are emphasized over small ones. \emph{Persistent homology}, an important tool in TDA, captures the evolution of the homology of a family of topological spaces associated to the given dataset. In the traditional setup, homology classes represented by geometrically larger cycles tend to \emph{persist} longer, or equivalently, have longer \emph{persistences}. They are given more emphasis because, on one hand, they tend to describe global structures of the dataset, and on the other, random variation of the data points tend to give rise to a large number of artificial small cycles.

\begin{figure}[t]
\centering
\includegraphics[trim = {2cm, 0cm, 2cm, 0cm}, clip, width = 0.45\linewidth]{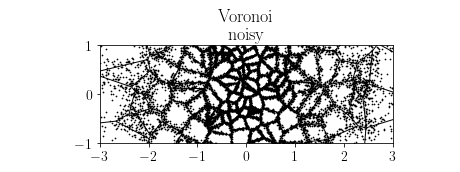}
\includegraphics[width = 0.45\linewidth]{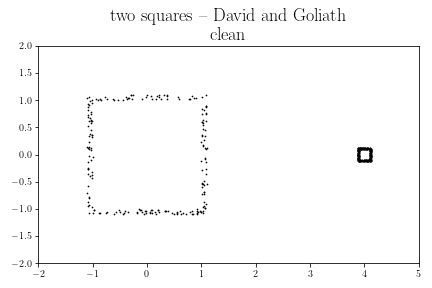}

\caption{Datasets with small topological features.}
\label{fig:TDA_illustration_small_features}
\end{figure}

However, small holes could be relevant too. For instance, in the toy example on the right of \cref{fig:TDA_illustration_small_features}, points are sampled from two squares with different sizes and densities. Since the smaller square has a higher density, it may be more relevant than the bigger square. Note that density consideration is crucial in the identification of small holes--in the extreme case, a ``small square'' formed by only four points is less likely a true hole than four points that happen to be nearby. The aforementioned toy example will be discussed further in \cref{sec:properties_small_dense_holes}. Beyond toy examples, small holes have been found to be relevant in practice too. They could be signs of enclave communities in network analysis \cite{stolz17_smallFeatures_networks, feng21_smallFeatures_voting}; or evidence of fractal structures or high-curvature regions \cite{jaquette20_smallFeatures_fractalDimension, bubenik20_smallFeatures_curvature}. Some datasets may only have small holes, or have holes with a wide range of sizes \cite{motta18_smallFeatures_hexagonalLattices, xia14_smallFeatures_protein, aragoncalvo12_cosmic_void_hierarchy, wilding21_cosmic_void_TDA}. Sometimes small features have better predictive power \cite{bendich16_brainArteryTrees}.

In the TDA literature, various approaches have been proposed to handle datasets with varying scales, and hence handle their smaller features more fairly. The multi-parameter-persistence approach has gained a lot of attention lately \cite{carlsson09_multipersistence, sheehy12_multicover, blumberg21_bipersistence_stability, lesnick15_interactive_bipersistence, moon18_persistence_terrace}. 
Alternatively, conformal geometry can be used to magnify small holes. For instance, the \emph{continuous $k$-NN graph} \cite{berry19_kNN_graph} and \emph{density-scaled Vietoris and Cech complexes} \cite{hickok22_density_scaled_filtration} were proposed, and the latter was shown to be stable against uniformly bounded extrinsic perturbation.

In the present work, a novel robust single-parameter scale-invariant approach, called the \emph{Robust Density-Aware Distance (RDAD) Filtration} is proposed to identify small holes. The proposed method scales the ambient Euclidean distance with density in the sense of Bell et al \cite{bell19_weighted_persistence} (see \cite{berry19_kNN_graph, hickok22_density_scaled_filtration} for conformal interpretations of such scaling) and robustness is ensured by incorporating the concept of \emph{distance-to-measure (DTM)}, which was first proposed in \cite{chazal11_DTM} and was further studied and generalized in \cite{buchet16_DTM_approximation, chazal18_DTM_statistics,anai19_DTM_generalizations}.

The proposed approach magnifies small holes of dense regions. Under suitable conditions, if a low-density region is surrounded by a high-density region, it gives rise to a homology class whose persistence scales with both a positive power of the density level and the size of the low-density region. Therefore, even if the low-density region is small, its persistence can still be large if the surrounding region  has a high enough density. This is made precise in \cref{cor:codim_1_features}.

The proposed approach is scale-invariant, in the sense that the persistences of all homology classes remain the same when the dataset is uniformly scaled. Therefore, the persistences of the topological features of the dataset are not reduced no matter by how much the dataset is shrunk. This is made precise in \cref{prop:scale_invariance}.

The proposed approach is robust. If a density is perturbed morderately in the Wasserstein metric and in the sup-norm, then the persistences of all homological features are also perturbed moderately. In particular, the proposed approach is provably robust against additive noise (with possibly unbounded support) and outliers. This is made precise in \cref{thm:robust} and \cref{cor:robust}.

These properties will be illustrated through variations of the toy example on the right subplot of
\cref{fig:TDA_illustration_small_features}. The more complicated synthetic dataset on the left subplot will then be studied. We also study a dataset of the locations of American cellular towers.

The rest of the paper is organized as follows. After reviewing the mathematical background in \cref{sec:background}, we define the proposed filtration in \cref{sec:proposed_filtration} and discuss its properties  in \cref{sec:properties}. We discuss bootstrapping in \cref{sec:bootstrapping} and present numerical simulations in \cref{sec:simulations}. A discussion and the conclusion are presented in \cref{sec:discussion,sec:conclusion}.
We collect proofs in \cref{sec:proofs} and simulation variables in \cref{sec:simulation_parameters}. An implementation of the proposed method is available at \url{https://github.com/c-siu/RDAD}.

\section{Background}
\label{sec:background}

We first review the theory of persistent homology in \cref{sec:persistent_homology}. In \cref{sec:specific_filtrations}, we discuss the distance filtration, which is the backbone of the traditional TDA approach, as well as various alternatives to the distance filtration, and we will explain their relevance to the present work. We conclude this section with a brief review of the theory of density estimation in \cref{sec:review_density_estimation}.

\subsection{Persistent Homology}
\label{sec:persistent_homology}

Persistent homology captures the evolution of the homology of a family of topological spaces. We review general definitions in this subsection and we discuss specific filtrations in the next.
We refer the reader to  \cite{edelsbrunner10comptopo,otter17_persistent_homology} for detailed expositions.

\paragraph{Filtrations and Persistence Diagrams}

A \emph{filtration} is a family $(X_t)$ of topological spaces such that 
$$X_s \subseteq X_t \text{ whenever } s \leq t.$$
The persistent homology of a filtration $(X_t)$ is the family of homology groups $(H_*(X_t))$ along with the maps between them induced by inclusion. Homology classes appear and vanish as the parameter $t$ varies, and the parameters at which they do so are called the \emph{birth time} and the \emph{death time} of the class. The class's \emph{persistence} is the difference of its birth and death times. The death time is infinite if the class never vanishes.

These birth and death times can be succinctly summarized in 
\emph{persistent diagrams}. The $k^{th}$ persistence diagram is a multiset of points in the extended quadrant $[0, \infty) \times [0, \infty]$. The $x$- and $y$-coordinates of each point in this multiset is the birth and death times of a $k$-dimensional homology class.

\paragraph{Stability, Interleaving and Bottleneck Distance}

Persistence diagrams are stable against certain perturbations of the filtration. This can be made precise using two similarity measures: the \emph{interleaving distance} for filtrations and the \emph{bottleneck distance} for persistence diagrams. Both similarity measures are symmetric and satisfy the triangle inequality. The map from filtrations to persistence diagrams is stable in the sense that the change in the persistence diagrams measured with bottleneck distance is bounded from above by the change in the filtration measured by the interleaving distance.
The two similarity measures are defined as follows.

Two filtrations $(X_t)$ and $(Y_t)$ are said to be \emph{$\varepsilon$-interleaved} if
\begin{equation}\label{eqn:interleaving}
X_{t} \subseteq Y_{t + \varepsilon} \text{ and } Y_{t} \subseteq X_{t + \varepsilon}
\end{equation}
for every parameter $t$. The interleaving distance of two filtrations is the infimum of all $\varepsilon$'s for which the two filtrations are $\varepsilon$-interleaved.

The bottleneck distance $W_\infty(P, Q)$ between persistence diagrams $P$ and $Q$ is the minimal $\Delta \geq 0$ such that there exists a ``bijective" pairing of points in $P$ and $Q$ with the sup-norm distance of the points in each pair bounded above by $\Delta$. ``Bijective" is in quotes because in general, the two diagrams may not have the same number of points, and excessive points are allowed to be paired with points on the diagonal $\{(x, x): x \in \mathbb{R}\}$. In symbols, we have
$$W_\infty(P, Q) = \inf_{\varphi: P \to Q \text{``bijective"}} \sup_{p \in P} \|p - \varphi(p)\|_\infty.$$
For points with infinite death time, we adopt the convention that $\infty - \infty = 0$ and $\infty - x = \infty$ for $x \in \mathbb{R}$. 

The bottleneck distance metric ball is useful for discerning significant features. In \cite{fasy14_confidence_set} as well as in the present work, a confidence set of persistence diagrams is defined as the bottleneck distance metric ball whose center is the empirical persistence diagram and whose radius $r$ is the significance threshold determined by bootstrapping. A homology class is considered significant if and only if the corresponding point $p$ in the empirical diagram lies above the line $y = x + 2r$, where $r$ is the significance threshold obtained by bootstrapping, because any diagram in the ball must have a non-diagonal point paired with the point $p$ in the empirical diagram.

\paragraph{Sublevel Filtration}

Given a function $f: X \to \mathbb{R}$ on a topological space $X$, its \emph{sublevel filtration} consists of the \emph{sublevel sets} $$f^{-1}(-\infty, t] = \{x : f(x) \leq t\}$$
of $f$. We often abuse the terminology and identify the sublevel filtration of a function with the function itself. All filtrations in the present work are sublevel filtrations.

Every pair of continuous functions $f, g: X \to \mathbb{R}$ on a compact space $X$ is $\|f - g\|_{L^\infty(X)}$-interleaved. Hence functions close together in the sup-norm have similar persistence diagrams. In the present work, all convergence results are locally uniform, and hence the persistence diagrams of the corresponding functions (when restricted to a compact set) become similar.

\subsection{Specific Filtrations}
\label{sec:specific_filtrations}

The most commonly used filtration is the distance filtration. While it can identify clean global topological signals, it is less useful for small and noisy features. To overcome this, multiple alternatives have been suggested. In this subsection, after briefly discussing the distance filtration, we review Bell et al's weighted filtration \cite{bell19_weighted_persistence} and the distance-to-measure filtration \cite{chazal11_DTM, buchet16_DTM_approximation, chazal18_DTM_statistics, anai19_DTM_generalizations}. The proposed filtration adapts the weighted filtration with density as the weight, and it incorporates the concept of distance-to-measure to enhance robustness.

\subsubsection{The Distance Filtration}

The distance filtration of a compact set $K$ in a metric space is the sublevel filtration of $d_K(x) = \min_{\xi \in K} d(x, \xi)$. If $K$ is finite, the sublevel sets are unions of metric balls, which grow in size as the filtration parameter increases. The distance filtration of data points in a Euclidean space is one of the most commonly used filtrations in topological data analysis. See \cref{fig:TDA_illustration_cech} and its caption for a simple example.

\begin{figure}[t]
\centering
\includegraphics[width=0.22\linewidth]{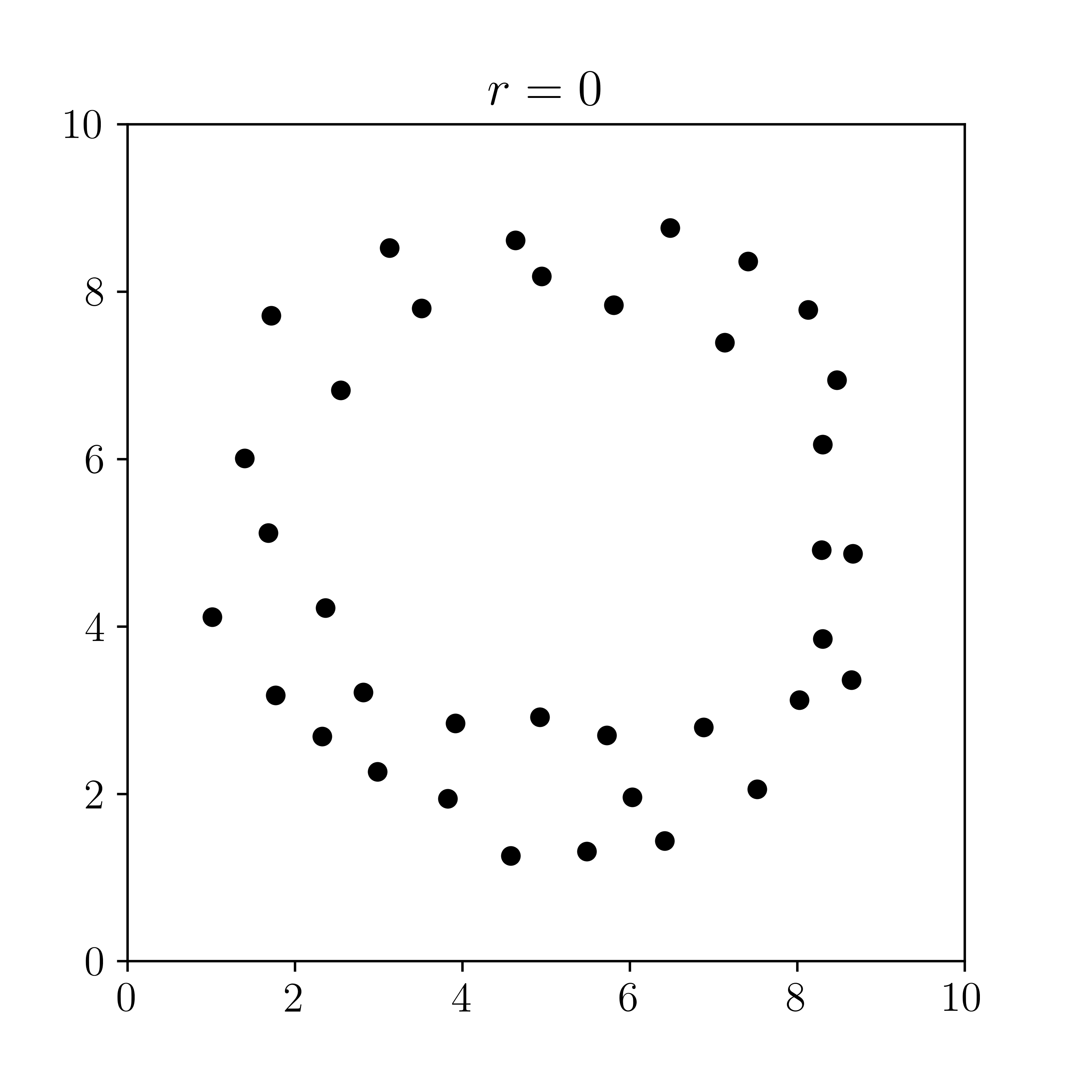}
\includegraphics[width=0.22\linewidth]{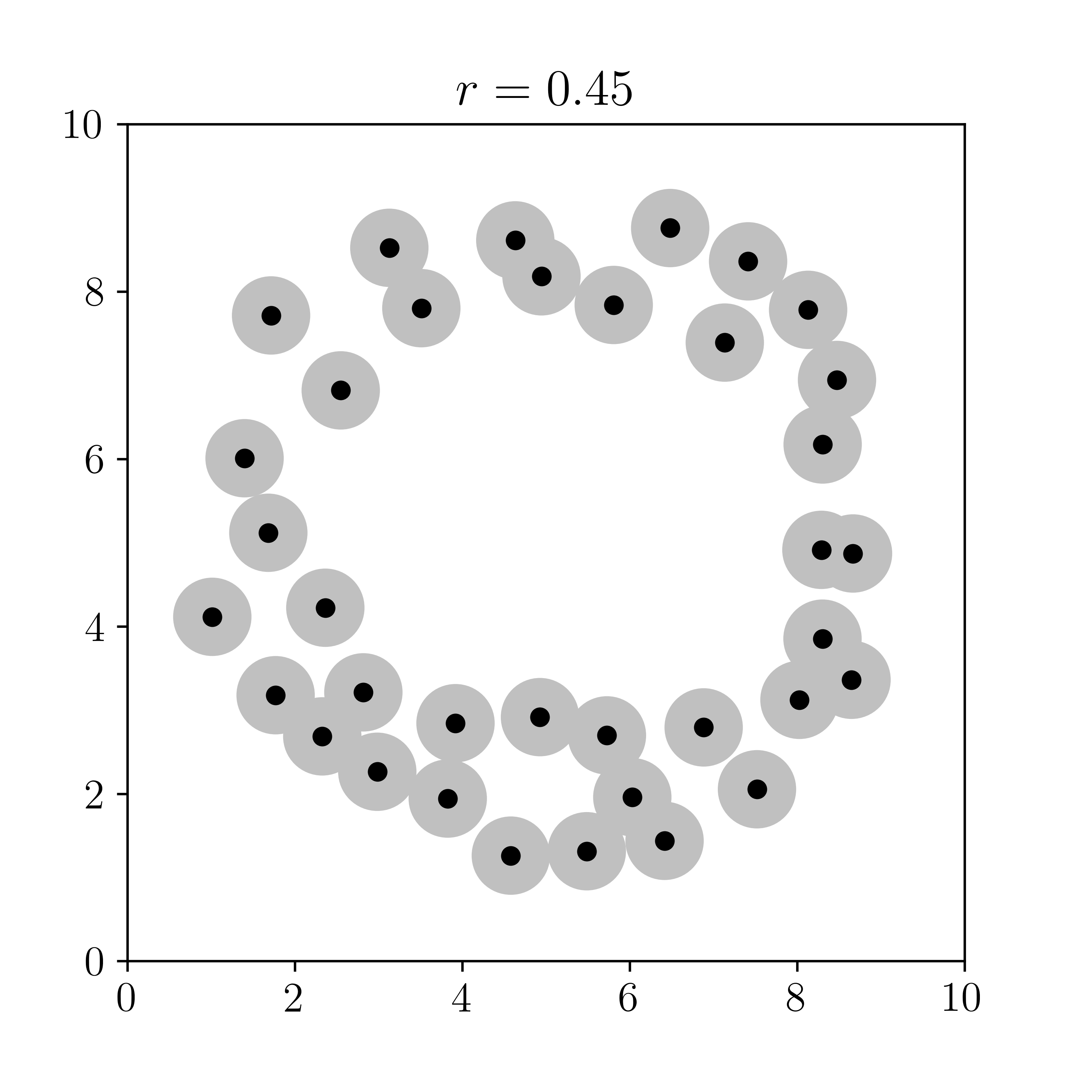}
\includegraphics[width=0.22\linewidth]{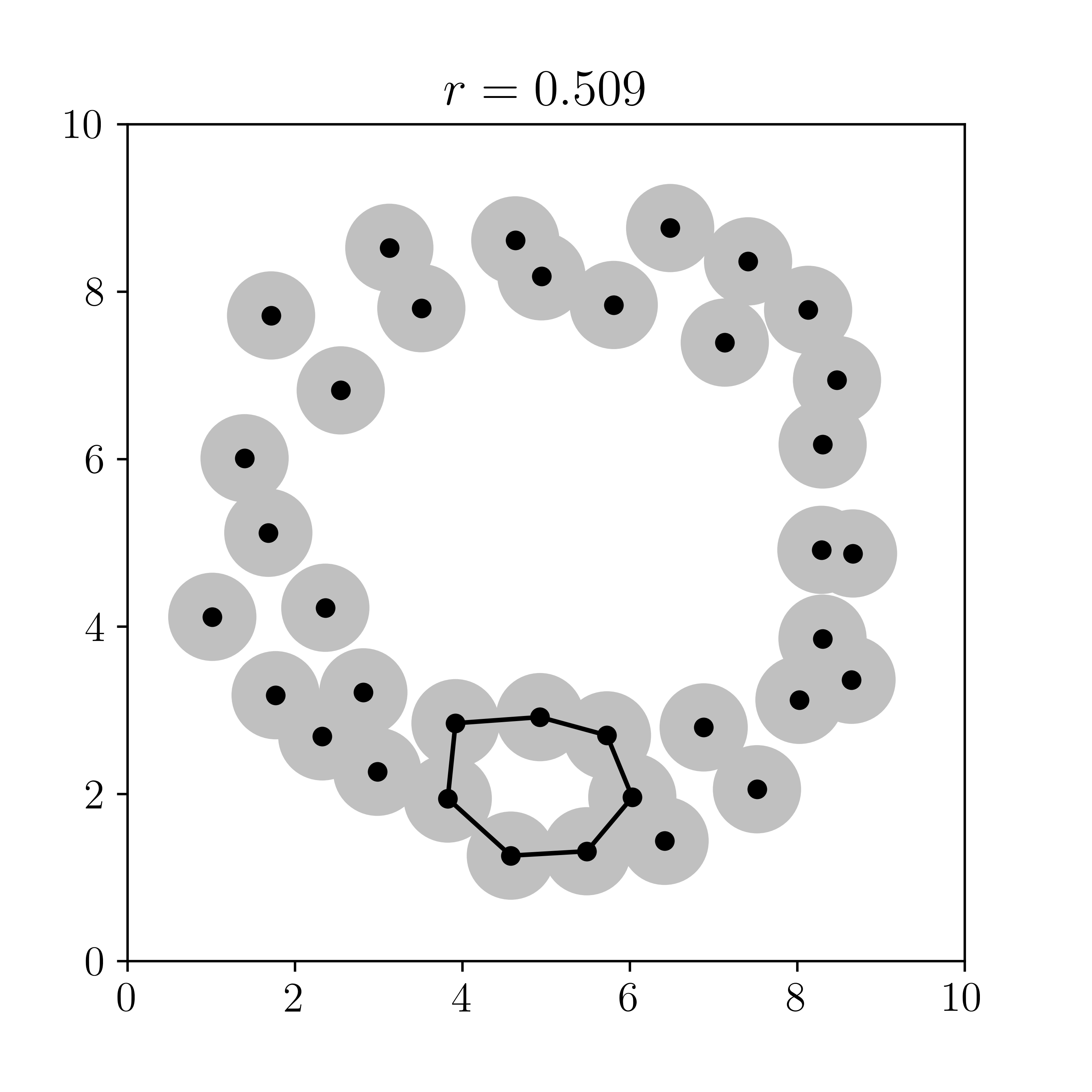}
\includegraphics[width=0.22\linewidth]{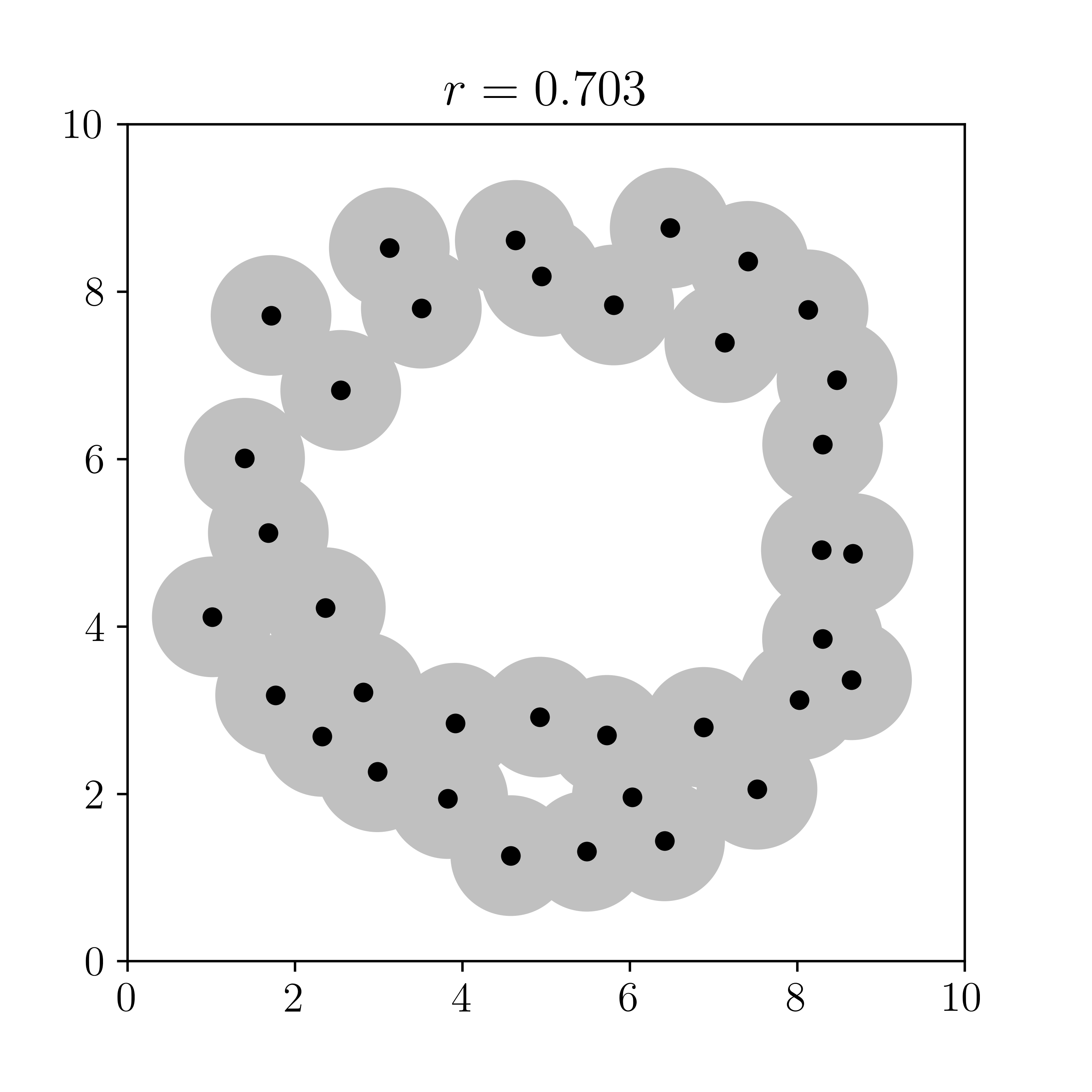}

\includegraphics[width=0.22\linewidth]{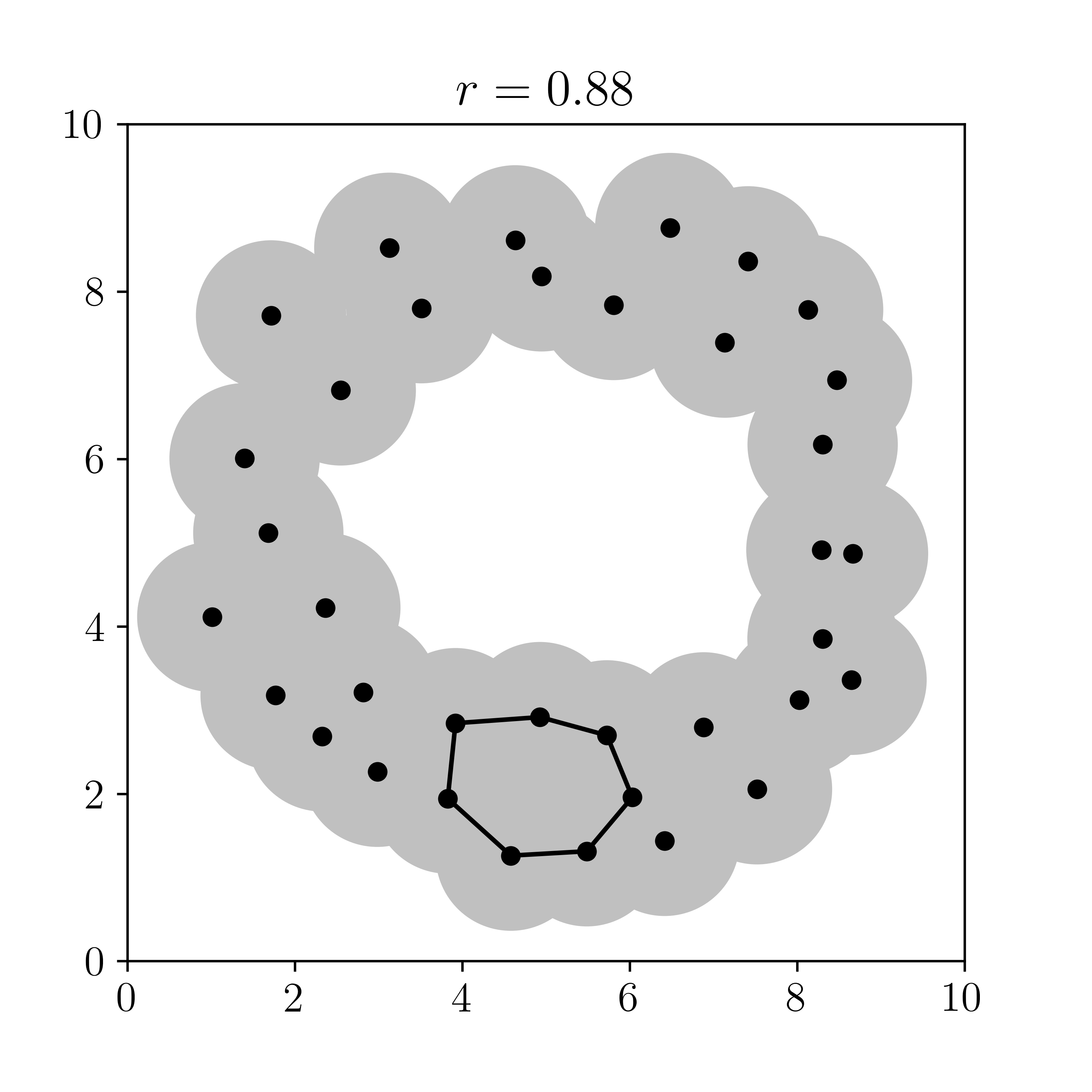}
\includegraphics[width=0.22\linewidth]{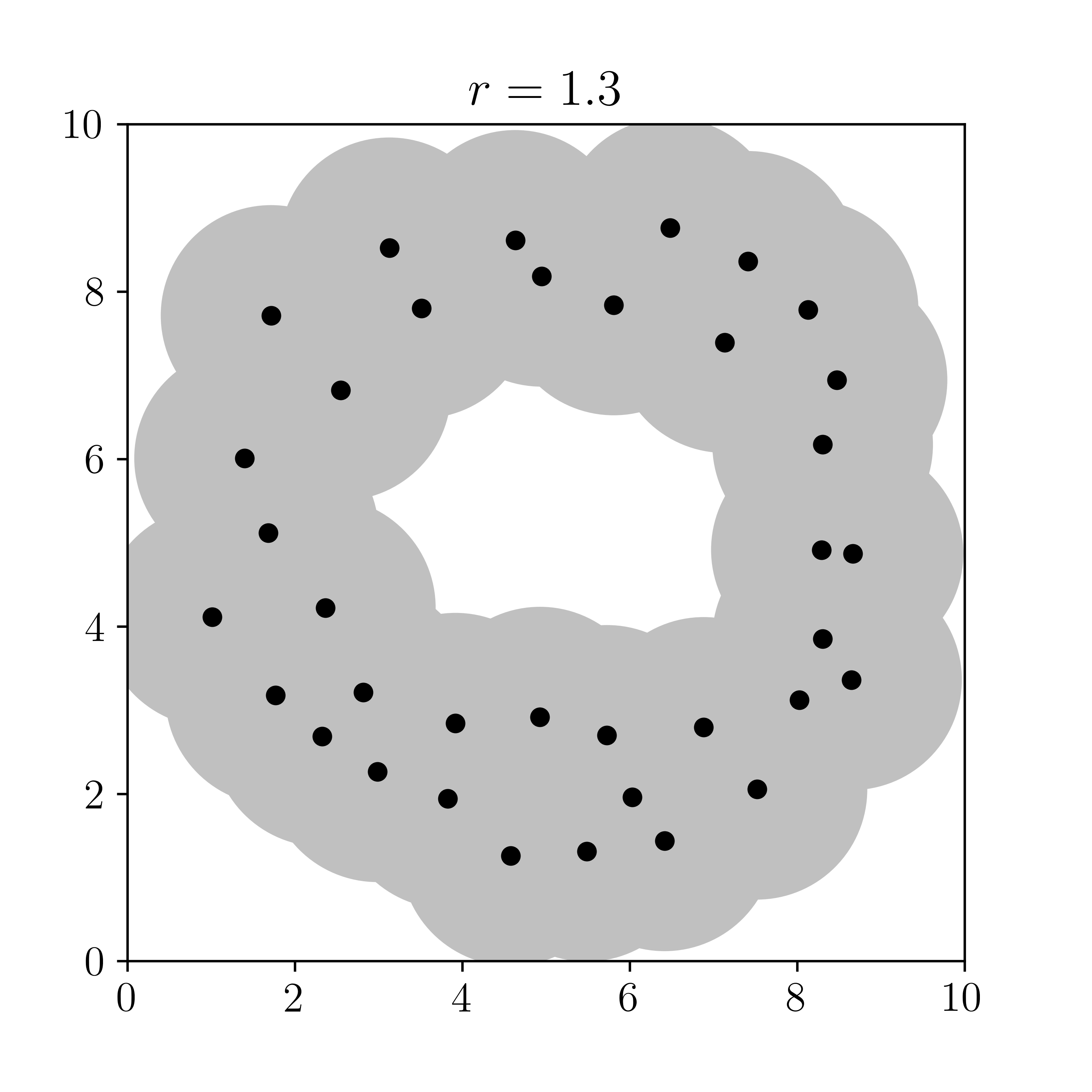}
\includegraphics[width=0.22\linewidth]{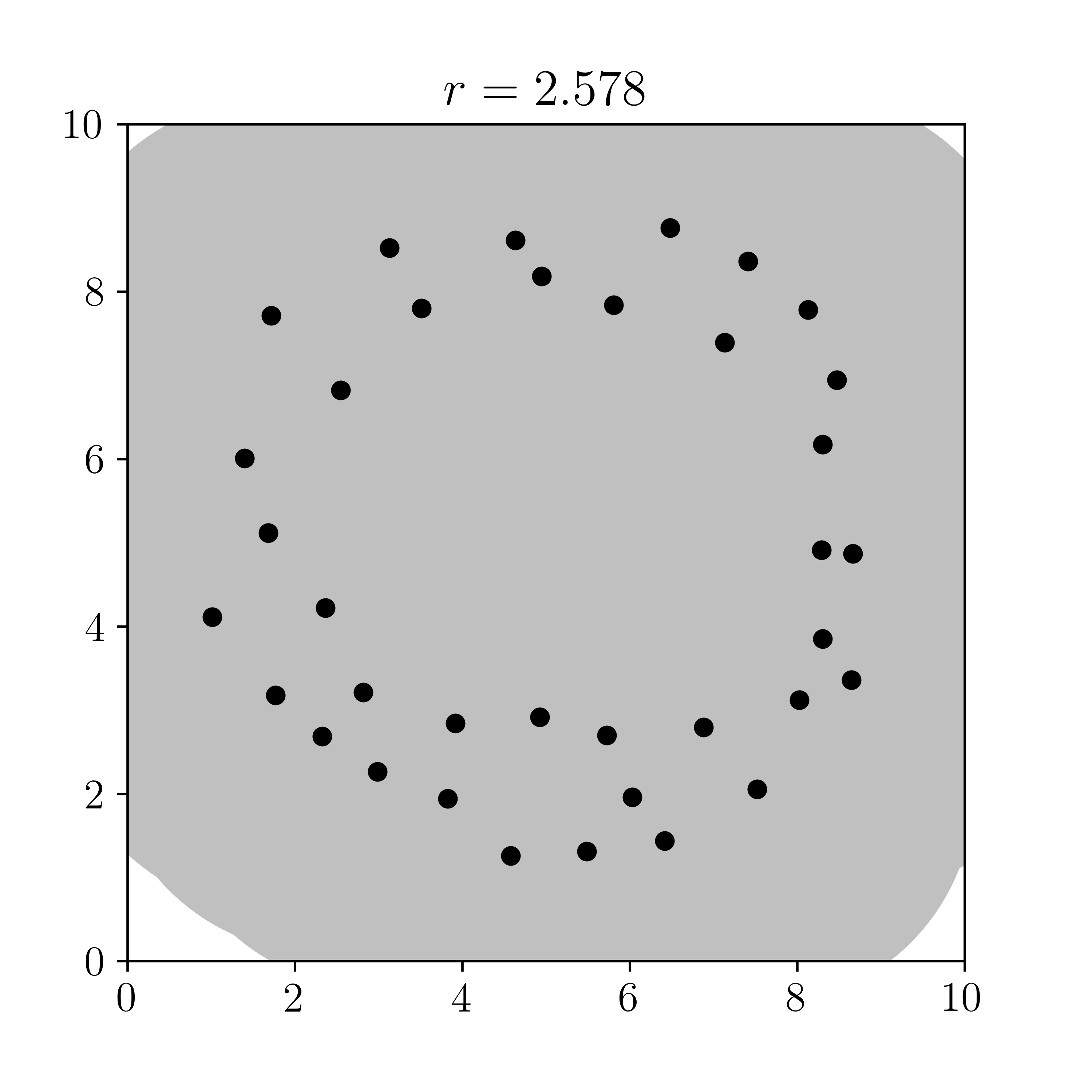}
\includegraphics[width=0.22\linewidth]{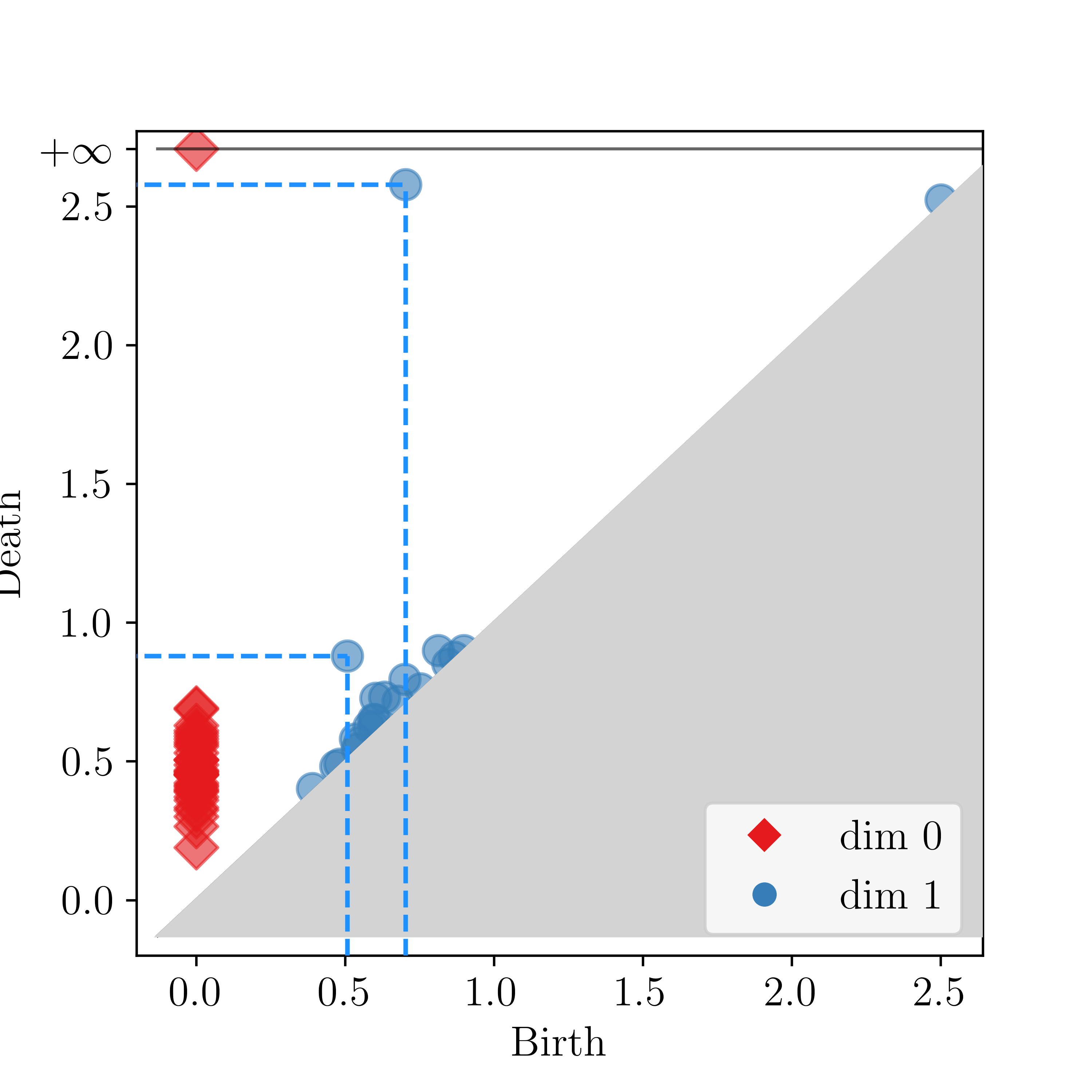}
\caption{The distance filtration and its persistence diagrams. In the first subplot is a sample of points near a circle. Unions of balls centered at these points with different radii are shown the subsequent subplots. The last subplot shows the persistence diagrams of these unions of balls. The red diamond points correspond to the dimension-0 diagram, and the blue circular points correspond to the dimension-1 diagram. The point marked by dashed lines near the diagonal corresponds to the marked polygon in the third subplot and is filled in the fifth subplot. The marked blue point that is far away from the diagonal corresponds to the main loop that is formed in the fourth subplot and is filled in the second last subplot.
}
\label{fig:TDA_illustration_cech}
\end{figure}

\subsubsection{Bell et al's Weighted Filtration}
\label{sec:bell_weighted_filtration}

In \cite{bell19_weighted_persistence}, a weighted filtration, based on the idea of growing balls at custom rates, is proposed.
Given points $\{X_1, ..., X_N\}$ and rates $v_1, ..., v_N > 0$,
$$\hat V_r = \cup_i \bar B(X_i, v_i r),$$
is considered. ($\bar B(x, s)$ denotes the closed metric ball with center $x$ and radius $s$.)
$\hat V_r$ is the sublevel filtration of the function $\hat d_v: \mathbb{R}^D \to \mathbb{R}$ defined by
$$\hat d_v(x) = \min_i d(x, X_i)/v_i.$$

The number of points needs not be finite. Given $E \subseteq \mathbb{R}^D$ and $v: E \to (0, \infty)$, one may consider 
$$V_r = \overline{\cup_{\xi \in E} B(\xi, v(\xi)r)}.$$
Under mild assumptions, e.g. when there exist constants $v_1, v_2$ such that $0 < v_1 \leq v(\xi) \leq v_2 < \infty$, $V_r$ is the sublevel filtration of
\begin{equation}\label{eqn:bell_weighted_filtration_oracle}
d_v(x) = \inf_{\xi \in E} d(x, \xi)/v(\xi).
\end{equation}

The paper \cite{bell19_weighted_persistence} establishes various combinatorial properties of the weighted filtration, which will not be needed in the present work, as the proposed filtration is implemented as a cubical filtration.

In the application presented in \cite{bell19_weighted_persistence}, the rates are chosen to be the pixel intensities of an image, and applications to noisy datasets are alluded to. In the present work, we will specialize to the case when $v$ is a function of the density from which the sample points $X_1, ..., X_N$ are drawn. However, the direct adaptation struggles with low-density regions, and hence is not robust against noise and outliers, as we will see in the simulations in \cref{sec:properties_robustness}. 
To develop a more nuanced approach for handling low-density regions, we borrow the idea of distance-to-measure.

\subsubsection{Distance-to-Measure (DTM) Filtration and Robustness}
\label{sec:DTM}

The distance-to-measure (DTM) function is a modification of the distance function that is designed to enhance robustness against potential noise and outliers.
Roughly speaking, the distance-to-measure of a point $x$ to a probability measure $\mu$ is the average distance of $x$ from the nearest part of the support that carries sufficient mass. As opposed to the distance to the support of $\mu$, it is not estimated by a minimum. Rather, it is averaged over a positive mass, and hence it is more robust. The distance-to-measure filtration is the sublevel filtration of the distance-to-measure function.

Formally, the distance-to-measure function of a probability measure $\mu$ on $\mathbb{R}^D$, with parameter $0 < m < 1$, is defined by
\begin{equation}\label{eqn:DTM_oracle}
DTM(x)
= DTM(x; \mu, m)
= \sqrt{\frac{1}{m} \int_0^{m} G_x^{-1}(q)^2 dq},
\end{equation}
where $G_x^{-1}$ is the generalized inverse function of
\begin{equation}\label{eqn:DTM_quantile}
G_x(r) = \mu[\bar B(x, r)].
\end{equation}
$G_x^{-1}(q)$ may be seen as the $q^{th}$ quantile of the distribution of $d(x, X)$, where $X$ is a random vector with distribution $\mu$.

When the probability measure takes the form $\hat \mu = \frac{1}{N} \sum_{i = 1}^N \delta_{X_i}$, where $\delta_x$ denotes the dirac delta measure at $x$, and $m = k_\text{DTM}/N$ for some positive integer $k_\text{DTM}$ smaller than $N$, the function takes the form
\begin{equation}\label{eqn:DTM_empirical}
DTM(x; \hat \mu, k_\text{DTM}/N)
= \sqrt{\frac{1}{k_\text{DTM}} \sum_{i=1}^{k_\text{DTM}} d(x, X_{(i)})^2},
\end{equation}
where $X_{(i)}$ is the $i^{th}$ nearest neighbor to $x$ among $X_1, ..., X_N$.

When $k_\text{DTM} = 1$, we 
recover the distance function 
to the set $\{X_1, ..., X_N\}$
as a special case. The distance-to-measure function is more robust against noise and outliers than the plain distance function because it takes into account not just \emph{the} nearest neighbor, which may be an outlier, but additional points as well.

In the present work, this idea will be used to enhance the robustness of the weighted filtration.

\subsection{Density Estimation by Nearest Neighbors}
\label{sec:review_density_estimation}

For a filtration $(Y_t)$ defined in terms of a density function $f$, we must first estimate $f$ before we can estimate the persistence homology of $(Y_t)$ from a dataset drawn from $f$. Common density estimation methods include the kernel method and the nearest neighbor method. The latter is used in the proposed filtration because it adapts the amount of smoothing to the local density. We refer the reader to \cite{silverman_density} for the theory of density estimation, and \cite{biau_devroye_nearest_neigbhor} for the nearest neighbor approach.

The nearest-neighbor density estimate
based on a sample $\{X_1, ..., X_N\} \subseteq \mathbb{R}^D$ is defined by
\begin{equation}
\label{eqn:NNDE}
\hat f_k(x) = \hat f_k(x; X_1, ..., X_N, k) = \frac{k}{N} \frac{1}{\omega_D d_k(x)^D},
\end{equation}
where
$\omega_D$ is the volume of the unit ball in $\mathbb{R}^D$ and 
$d_k(x)$ is the distance from $x$ to the $k^{th}$ nearest neighbor of $x$ among the points $X_1$, ..., $X_N$, i.e. 
$$d_k(x) = d(x, X_{i_k}), \text{ where}$$
$$d(x, X_{i_1}) \leq d(x, X_{i_2}) \leq ... \leq d(x, X_{i_N}).$$

The estimate \cref{eqn:NNDE} is motivated by the local approximation
\begin{align*}
\frac{k}{N}
\approx \int_{B(x, d_k(x))} f(y) dy
\approx f(x) \lvert B(x, d_k(x)) \rvert = f(x) \omega_D d_k(x)^D,
\end{align*}
where
$\lvert \cdot \rvert$ denotes the Lebesgue measure, and $f$ is the true density of the independent and identically distributed sample $\{X_1, ..., X_N\}$.

The first approximation 
applies when $N$ is large, since the empirical measure of the ball approximates its true measure. The second approximation applies when $f$ is smooth enough near $x$ and when $d_k(x)$ is small.

\section{Proposed Filtration}
\label{sec:proposed_filtration}

We define the proposed filtration and its estimator in this section.

\begin{definition}[Population RDAD]\label{def:RDAD_population}
Let $f$ be a density
on
$\mathbb{R}^D$ and $P$ be the measure induced by $f$. 
Let $m \in (0, 1)$. The population robust density-aware distance function $RDAD$ is defined by
\begin{align}\label{eqn:RDAD_oracle}
RDAD(x) &= RDAD(x; f, m)\notag
\\&= \sqrt{\frac{1}{m} \int_0^{m} F_x^{-1}(q)^2 dq},    
\end{align}
where
\begin{equation}\label{eqn:RDAD_quantile}
F_x(r) = P \left[
f(X)^{1/D}
d(X, x) \leq r \right].
\end{equation}
\end{definition}

For the sake of comparison, the Density-Aware Distance ($DAD$) function, which has been shown to be not robust against noise and outliers in \cite{hickok22_density_scaled_filtration}, is defined as follows.

\begin{definition}[Population DAD]
Let $f$ be a density
on
$\mathbb{R}^D$ and $P$ be the measure on $\mathbb{R}^D$ induced by $f$. 
The population density-aware distance function $DAD$ is defined by
\begin{equation}\label{eqn:weighted_distance_oracle}
DAD(x) = \text{ess-} \inf_{y \in \mathbb{R}^D} d(x, y) f(y)^{1/D},
\end{equation}
where $\text{ess-} \inf$ denotes the essential infimum with respect to the measure $P$.
\end{definition}

One may check that the DAD function is Bell et al's weighted filtration with $1/f(x)^{1/D}$ as the growth rates of metric balls, and RDAD is the DAD function made robust using the DTM idea. Intuitively, weighting by the density slows down ball growth in high-density regions, so that persistences of homology classes of those regions are prolonged. We make this idea precise in \cref{sec:properties_small_dense_holes}.

We now introduce an empirical version of the RDAD function and the DAD function that are defined in terms of a sample $\{X_1, ..., X_N\} \subseteq \mathbb{R}^D$. The empirical RDAD is defined by combining the distance-to-measure function \cref{eqn:DTM_empirical} for the discrete measure $\frac{1}{N} \sum_{i=1}^N \delta_{X_i}$ in \cref{sec:DTM} and the nearest neighbor density estimator. This requires two parameters, $k_\text{DTM}$ and $k_\text{den}$. The former is needed to set $m = k_\text{DTM}/N$, and the latter is the number of nearest neighbors used for density estimation. One gets the empirical DAD function by putting $k_\text{DTM} = 1$.

Let $X_1, ..., X_N \in \mathbb{R}^D$, and for each $i = 1, ..., N$, let $d_i$ be the distance from $X_i$ to its $k_\text{den}^{th}$-nearest neighbor among $X_1, ...., X_N$.

\begin{definition}[Empirical DAD and RDAD]
Let $k_\text{DTM}, k_\text{den}$ be positive integers strictly less than $N$. 
For each $x \in \mathbb{R}^D$, let $d(x, X_{(i)})/d_{(i)}$ be the $i^{th}$ order statistic of $d(x, X_1)/d_1$,..., $d(x, X_N)/d_N$. The empirical (robust) density-aware distance functions $\widehat{DAD}$ and $\widehat {RDAD}$ are defined by
\begin{align}
\widehat{DAD}(x)
&= \widehat{DAD}(x; X_1, ..., X_N, D, N, k_\text{den}) \notag
\\&= \min_i d(x, X_{i}) \hat f_{k_\text{den}}(X_{i})^{1/D}
\label{eqn:DAD_empirical}
\\&= \min_i C_{N, k_\text{den}, D} d(x, X_{i})/d_{i} \notag
\\\widehat {RDAD}(x) 
&= \widehat{RDAD}(x; X_1, ..., X_N, D, N, k_\text{den}, k_\text{DTM}) \notag
\\&= C_{N, k_\text{den}, D}\sqrt{\frac{1}{k_\text{DTM}}\sum_{i =1}^{k_\text{DTM}} (d(x, X_{(i)})/d_{(i)})^2}
\label{eqn:RDAD_empirical}
\end{align}
where $C_{N,k_\text{den},D} = \left( \frac{1}{\omega_D} \frac{k_\text{den}}{N} \right)^{1/D}$
and $\omega_D$ is the volume of the unit ball in $\mathbb{R}^D$.
\end{definition}

\section{Properties}
\label{sec:properties}

In this section, we present certain desirable properties of the proposed filtration and substantiate our claims in the introduction. In \cref{sec:properties_small_dense_holes}, we discuss how the proposed filtration prolongs persistences of homology classes of high-density regions. Then we discuss, in \cref{sec:properties_scale_invariance}, the proposed filtration's scale invariance, which motivates the awkward-looking exponent $1/D$ in the definition of the RDAD function, and in \cref{sec:properties_robustness}, its robustness, which is enhanced by the DTM setup. We conclude by giving further mathematical properties of the proposed filtration in \cref{sec:properties_further}. All proofs are delayed to \cref{sec:proofs}.

\subsection{Prolonged Persistence for High-Density Regions}
\label{sec:properties_small_dense_holes}

In this subsection, we illustrate how the proposed filtration prolongs persistences of homology classes of high-density regions with a numerical example, and we formalize the observations from the example with theorems. For the numerical examples in this and subsequent subsections, parameters are summarized in \cref{tab:bisquare_variables_setup} in \cref{sec:simulation_parameters}, and implementation details are deferred to \cref{sec:simulations_implementation}.

\subsubsection{Example}

\begin{figure}[h]
\centering
\includegraphics[width = 0.45\linewidth]{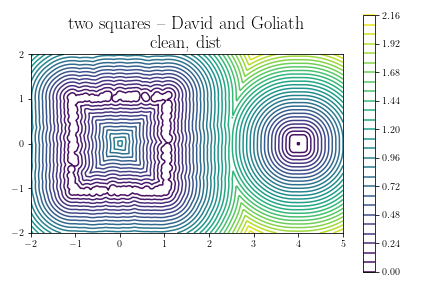}
\includegraphics[width = 0.45\linewidth]{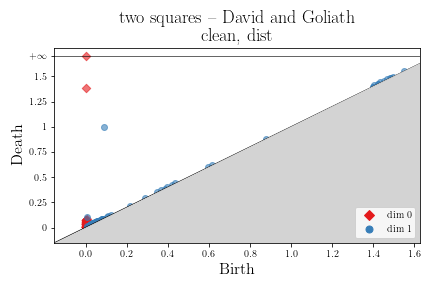}

\includegraphics[width = 0.45\linewidth]{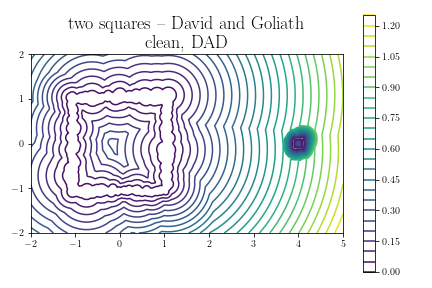}
\includegraphics[width = 0.45\linewidth]{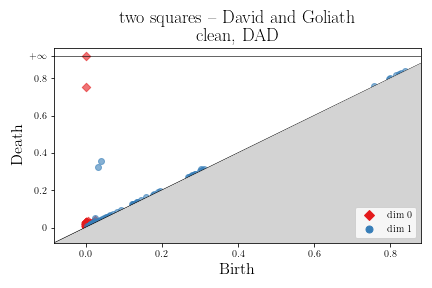}

\caption{Contour plots and persistence diagrams of different filtrations for the ``David and Goliath" two-square dataset.}
\label{fig:bisquareShadowed_clean_filtration_persistence}
\end{figure}

Recall the two-square dataset ``David and Goliath" in the right subplot of \cref{fig:TDA_illustration_small_features} in the introduction. 100 points are uniformly sampled from the bigger square annulus, and 400 from the smaller annulus. Since the dataset has no additive noise or outliers, we compare the distance filtration and the DAD filtration for this dataset. Their contour plots and persistence diagrams are shown in  \cref{fig:bisquareShadowed_clean_filtration_persistence}. The plots and the diagrams for DTM and RDAD are similar to the distance and DAD figures correspondingly. 

Two squares are clearly visible in the scatter plot in the right subplot of \cref{fig:TDA_illustration_small_features}. However, the blue point corresponding to the smaller square  in the persistence diagram of the distance filtration is very close to the diagonal (it is at the tip of the cluster of red diamonds near the origin). On the other hand, for the DAD filtration, two blue circular points in the persistence diagrams are comfortably far away from the diagonal. The contour plot of the DAD function explains this: the dense contour lines inside the small square, due to the higher density on the smaller annulus, show the smaller square hole does not get filled over a wide range of levels. This shows that the proposed choice of weights prolongs the persistence of homology classes of high-density regions.

\subsubsection{Setup and Assumptions}

We need the following setup for formal statements.

\begin{itemize}
\item $f: \mathbb{R}^D \to [0, \infty)$ is a bounded probability density function, and $P$ is the probability measure induced by $f$.
\item $0 \leq t_0 \leq ... \leq t_k \leq \|f\|_\infty$ and for each $i$, $\Omega_i$ is a connected component of $f^{-1} [t_i, \infty)$. The $\Omega_i$'s are pairwise disjoint.
\item $s_i = \min_{j \neq i} \frac{d(\Omega_i, \Omega_j)}{t_i^{-1/D} + t_j^{-1/D}}$.
\end{itemize}

For the two-square example above, the points are sampled from a piecewise constant density supported on the two square annuli, which we may take to be $\Omega_1$ and $\Omega_2$.

We also need the following mass concentration conditions.
\begin{description}
\item[Global Condition] $P(\mathbb{R}^D - \cup \Omega_i) \leq m/2$.
\item[Local Condition] 
There exist $a > 1$ and $\rho_m > 0$ such that for every $x \in \cup \Omega_i$,
\begin{equation}\label{eqn:local_mass_concentration}
P(B(x, \rho_m) \cap f^{-1}[0, a f(x)]) \geq m.
\end{equation}
\end{description}

It is not difficult to check that, for an $L$-Lipscthiz density $f$, if
$m$ is much smaller than $\left( \frac{t_0}{\max(L, 1)}\right)^{1/D}$, the local condition is satisfied for any $a$ when one set
$$\rho_m \sim \left(\frac{m}{\omega_D t_0}\right)^{1/D} \sim m^{1/D}.$$

Finally, we also assume $f \mid \partial \Omega_i = t_i$ for each $i$. Note that this condition and the local mass concentration condition above both concern pointwise function values of the density function $f$, which is only defined up to a set of measure zero. We adopt the convention that these conditions hold as long as they both hold for a function that is almost everywhere equal to $f$.

\subsubsection{Results}

All results of this section assume the setup laid out above.

We first show that the RDAD function is similar to the DAD function corresponding to a piecewise constant approximation of the density $f$.
Let
\begin{align*}
g_i(x) &= t_i^{1/D} d(x, \Omega_i), \quad i = 1, ..., k\\
g(x) &= \min_{i = 1, ..., k} g_i(x).
\end{align*}
In the example, since the density is piecewise constant, $g$ is precisely $DAD$.
Note that a sublevel set of $g_i$ is a neighborhood of $\Omega_i$ consisting of $\Omega_i$ and a surrounding band whose width is proportional to $1/t_i^{1/D}$. Hence, if $t_i$ is large, then persistences of homology classes in the filtration $g_i$ are prolonged by a factor of $t_i^{1/D}$. This remains true for $g$ as long as the filtration level is smaller than $s_i/t_i^{1/D}$, because the sublevel set of $g$ of each level is the disjoint union of the sublevel sets of the $g_i$'s at the level.

\begin{theorem}\label{thm:RDAD_interleaving}
For every $x \in \mathbb{R}^D$,
\begin{equation*}
\frac{1}{\sqrt{2}} g(x) \leq RDAD(x; f, m) \leq a^{1/D} g(x) + O(\rho_m),
\end{equation*}
where the big-Oh constant depends only only on $\|f\|_\infty, t_0, D, a$.
\end{theorem}

\begin{corollary}\label{cor:RDAD_persistence_lower_bound}
Every homology class in the distance filtration of $\Omega_i$ with birth and death times $\beta$ and $\delta$ induces a class in the RDAD filtration with persistence at least $$t_i^{1/D} (\frac{1}{\sqrt{2}}\delta - a^{1/D}\beta) - O(\rho_m)$$
whenever $t_i^{1/D} \delta < s_i$.
\end{corollary}

In particular, if 
the difference in the parentheses in the corollary is positive and $\rho_m$ is small, then the persistence is scaled roughly by a factor of $t_i^{1/D}$.

Returning to the central claim that the RDAD filtration highlights low-density regions surrounded by high-density regions, the above corollary combined with Alexander duality gives the following result at dimension $D-1$.

\begin{corollary}\label{cor:codim_1_features}
Suppose $f$ is $C^1$-smooth and $t_i$ is not a critical value of $f$. Let $M$ be a bounded connected component of the complement of $\Omega_i$ and $r = \max_{x \in M} d(x, \partial M)$. Then $\partial M$ determines a $(D-1)$-dimensional homology class in the RDAD filtration with persistence at least
$$\frac{1}{\sqrt{2}} t_i^{1/D} r - O(\rho_m)$$
whenever $t_i^{1/D} r < s_i$.
\end{corollary}

In the corollary, $M$ is a low-density region that is surrounded by the high-density region $\Omega_i$. Its boundary induces a $(D-1)$-homology class in the RDAD filtration whose persistence is at least a quantity that scales with a positive power of the density level and $r$, which measures the size of the low-density region $M$. Hence, if the density threshold $t_i$ or the size $r$ of $M$ is big, and the other factor is moderate, the persistence will be long as long as $\rho_m$ is not too big.

In the two-square example, $M$ could be the square inside the small annulus, and $r$ is the half the sidelength of $M$.

\subsection{Scale Invariance}
\label{sec:properties_scale_invariance}

While many other choices of growth rates may prolong the persistences of small features, our specific choice makes the filtration \emph{scale-invariant}. This means 
that uniformly scaling a dataset does not change the persistence diagrams.
In particular, no matter by how much a dataset is shrunk, its topological features still have the original persistences. Precisely, we have the following proposition.

\begin{proposition}[Scale invariance]\label{prop:scale_invariance}
Let $a > 0$ and $b \in \mathbb{R}^D$ be constants.
\begin{description}
\item[Population version:] Let $X$ be a random vector in $\mathbb{R}^D$ with density $f$. Let $\tilde X = aX + b$ and $\tilde f$ the density of $\tilde X$. Then for any $0 < m < 1$,
$$
RDAD(ax + b; \tilde f, m)
= RDAD(x; f, m),
$$
and hence the $RDAD(\; \cdot \;; \tilde f, m)$ and $RDAD(\; \cdot \;; f, m)$ have the same persistence diagrams.
\item[Sample version:] Let $X_1, ..., X_N$ be points in $\mathbb{R}^D$, and let $\tilde X_i = a X_i + b$. Then for any positive integer $k_\text{DTM}$ strictly less than $N$,
\begin{align*}
&\widehat{RDAD}(ax + b; \tilde X_1, ..., \tilde X_N, D, N, k_\text{den}, k_\text{DTM})
\\&\qquad =
\widehat{RDAD}(x; X_1, ..., X_N, D, N, k_\text{den}, k_\text{DTM}),
\end{align*}
and hence the two $\widehat{RDAD}$ filtrations have the same persistence diagrams.
\end{description}
\end{proposition}

Theorem 5 of \cite{hickok22_density_scaled_filtration} is a conformal analogue of this result.

We illustrate the scale invariance property with the ``Antman" example in \cref{fig:bisquare2D_clean_samplePoints_persistence}. The same number of points are sampled randomly from two square annuli, which are scaled versions of each other. Thus, the two square holes 
give two nearby (overlapping) blue circular points in the persistence diagrams in \cref{fig:bisquare2D_clean_samplePoints_persistence}. 
\begin{figure}[h]
\centering
\includegraphics[width = 0.45\linewidth]{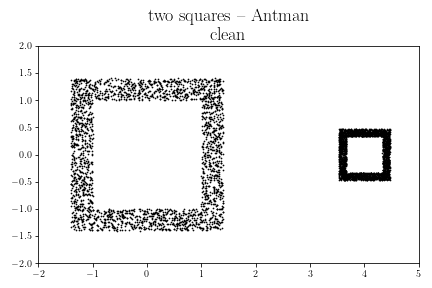}
\includegraphics[width = 0.45\linewidth]{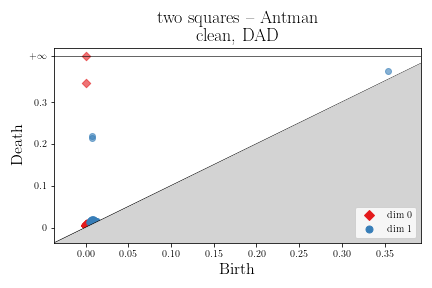}

\caption{Sample points of the ``Antman" two-square dataset, and the persistence diagram of the empirical DAD filtration for this dataset}
\label{fig:bisquare2D_clean_samplePoints_persistence}
\end{figure}

\subsection{Robustness}
\label{sec:properties_robustness}

The RDAD filtration is designed to be robust against additive noise and outliers. Indeed, \cref{thm:robust} below shows that the RDAD function is only mildly perturbed if the perturbation of $f$ is small in both the Wasserstein metric and the sup-norm relative to the parameter $m$. In particular, \cref{cor:robust} shows the perturbation is mild in the presence of additive noise and outliers.

We need the following notations and terminologies.
\begin{description}
\item[Wasserstein distance] By the order-$p$ Wasserstein distance $W_p(f, \tilde f)$ between two densities $f, \tilde f$, we mean the order-$p$ Wasserstein distance between the measures they induce.
\item[Moment] By the $p^{th}$ moment $M_p$ of a random vector $X$ we mean $M_p = E[\abs{X}^p]$. Moments of density functions and measures are defined analogously.
\item[Moderate Tail] A random vector $X$ in $\mathbb{R}^D$ is said to have a moderate tail with parameters $C, \alpha > 0$ if $P(\abs{X} > R) \leq CR^{-\alpha}$. The notion of moderate tails for density functions and measures are defined analogously.
\end{description}
We are now ready to state our stability result.

\begin{theorem}[Stability]\label{thm:robust}
Let $2 < p \leq \infty$, and let $f$ and $\tilde f$ be densities on $\mathbb{R}^D$. Suppose $f$ and $\tilde f$ have finite $q^{th}$ moments for every $q \in [1, \infty)$ and they have moderate tails with parameters $C, \alpha > 0$, where $\alpha \geq D^2 - D$. Suppose further that $f^{1/D}$ is Lipschitz and bounded.
If $2 < r < p$, $W_{p}(f, \tilde f) \leq 1$ and $\|f - \tilde f\|_\infty \leq 1$, then for any $0 < m < 1$ and every compact set $K$, 
$$\|RDAD(\; \cdot \; ; f, m) - RDAD(\; \cdot \; ; \tilde f, m)\|_{L^\infty(K)} =  \frac{1}{m^{1/r}} 
O(W_p(f, \tilde f) + \|f - \tilde f\|_\infty^{\frac{D+r}{(D+1)(r+1)}}),$$
where the big-Oh constant depends only on 
$f, D, K, C, p, r$ and the moments of $f$ and $\tilde f$.
\end{theorem}

To model additive noise and outliers, let $X, Y^\text{add}$ and $Y^\text{out}$ be random vectors in $\mathbb{R}^D$, with densities $f, g^\text{add}$ and $g^\text{out}$; and let $Z$ be a Bernoulli random variable with success probability $\delta \in [0, 1)$. Let $\varepsilon \in [0, 1)$. Suppose $X, Y^\text{add}, Y^\text{out}$ and $Z$ are independent. We model a corrupted $X$ by considering
$$\tilde X = (1-Z)(X + \varepsilon Y^\text{add}) + Z Y^\text{out}.$$

\begin{corollary}\label{cor:robust}
Suppose $X, Y^\text{add}$ and $Y^\text{out}$ have $q^{th}$ moments for every $q \in [1, \infty)$ and they have moderate tails with parameters $C, \alpha > 0$, where $\alpha \geq D^2 - D$. Suppose $f^{1/D}$ is Lipschitz and bounded, and $g^\text{out}$ is bounded. Then for any $r \in (2, \infty)$,
$$\|RDAD(\; \cdot \; ; f, m) - RDAD(\; \cdot \; ; \tilde f, m)\|_{L^\infty(K)} = \frac{1}{m^{1/r}}O(\delta^{1/(r+1)} + \varepsilon^{\frac{D+r}{(D+1)(r+1)}}),$$
where the big-Oh constant depends only on $f, D, K, C, r$, moments of $f, g^\text{add}$ and $g^\text{out}$, and their sup norms as functions on $\mathbb{R}^D$.
\end{corollary}

The DTM analogue of this theorem is Lemma 4 of \cite{chazal18_DTM_statistics}. Our proof loosely follows the argument there but there are considerable complications owing to the asymmetry of the definition of RDAD.

We illustrate the results above with corrupted versions of  the ``Antman" example in \cref{fig:bisquare2D_noisy_outlier_samplePoints}. We compare the DAD filtration and the RDAD filtration in \cref{fig:bisquare2D_outlier_filtration_persistence,fig:bisquare2D_noisy_filtration_persistence}. The persistence diagrams of RDAD for the corrupted datasets are affected to a lesser extent by the noise and outliers than those of DAD.

\begin{figure}[h]
\centering
\includegraphics[width = 0.45\linewidth]{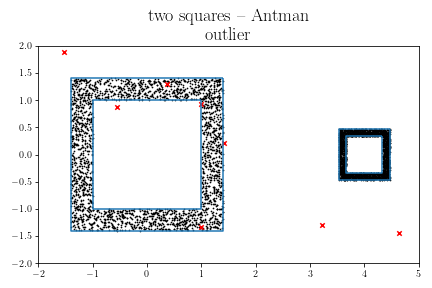}
\includegraphics[width = 0.45\linewidth]{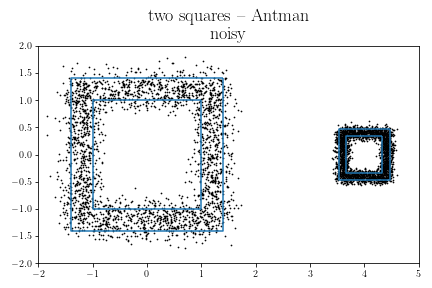}

\caption{Sample points of corrupted ``Antman" two-square datasets (by outliers and by additive noise).}
\label{fig:bisquare2D_noisy_outlier_samplePoints}
\end{figure}

\begin{figure}[h]
\centering
\includegraphics[width = 0.45\linewidth]{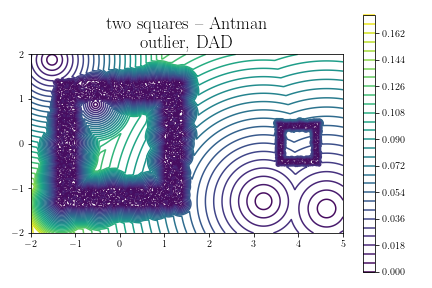}
\includegraphics[width = 0.45\linewidth]{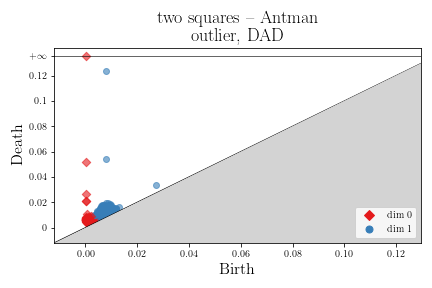}

\includegraphics[width = 0.45\linewidth]{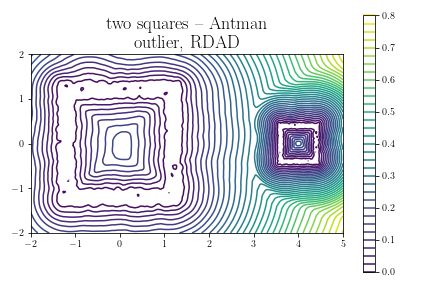}
\includegraphics[width = 0.45\linewidth]{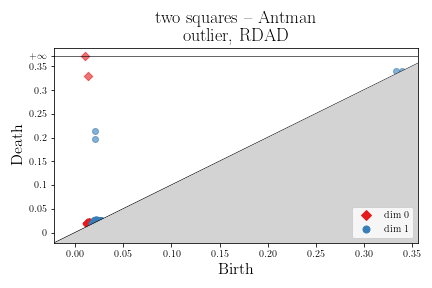}

\caption{Contour plots and persistence diagrams of different filtrations for the outlier-contaminated ``Antman" two-square dataset.}
\label{fig:bisquare2D_outlier_filtration_persistence}
\end{figure}

\begin{figure}[h]
\centering
\includegraphics[width = 0.45\linewidth]{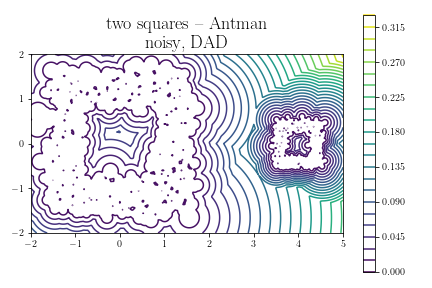}
\includegraphics[width = 0.45\linewidth]{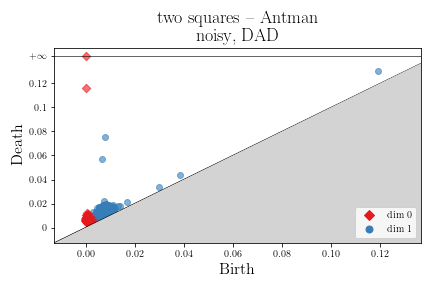}

\includegraphics[width = 0.45\linewidth]{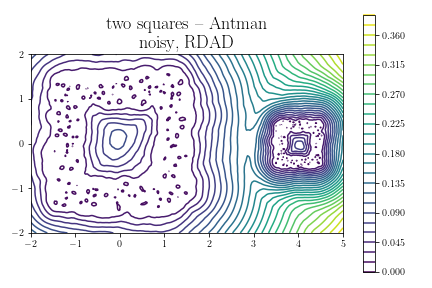}
\includegraphics[width = 0.45\linewidth]{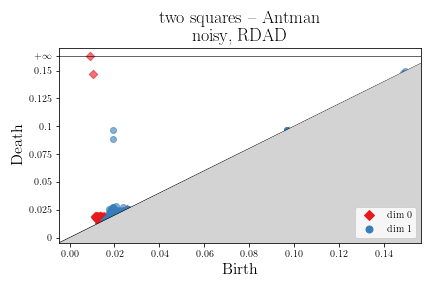}

\caption{Contour plots and persistence diagrams of different filtrations for the additive noise-contaminated ``Antman" two-square dataset.}
\label{fig:bisquare2D_noisy_filtration_persistence}
\end{figure}

\subsection{Further Properties}
\label{sec:properties_further}

We present basic mathematical properties of the proposed filtration in this subsection.

The first result extends Theorem 3.2 in \cite{chazal18_DTM_statistics} and shows RDAD is indeed an approximation of the DAD function.

\begin{proposition}[RDAD as an Approximation of DAD]
\label{prop:RDAD_approximates_DAD}
Given a density $f$, $RDAD(x; f, m) \to DAD(x; f)$ pointwise as $m \to 0$. The convergence is uniform on every compact set $K$ if $f$ is bounded.
\end{proposition}

The next lemma is useful for establishing results in previous subsections. It extends Proposition 3.3 of \cite{chazal11_DTM}.

\begin{lemma}[Variational Characterization]\label{lem:variational_characterization}
For each $x \in \mathbb{R}^D$,
\begin{equation*}
m \cdot RDAD(x)^2 = \min_{\substack{\nu(\mathbb{R}^D) = m\\\nu \text{ subordinate to } P}} \int [f(\xi)^{1/D} d(\xi, x)]^2 d\nu(\xi),
\end{equation*}
where a measure $\nu$ is said to be \emph{subordinate} to $P$ if $0 \leq \nu(E) \leq P(E)$ for every measurable set $E$. The minimum is attained by definition.
\end{lemma}

Lipschitz continuity of RDAD ensures that it can be numerically approximated. Its DTM analogue is Theorem 3.1 of \cite{chazal18_DTM_statistics}.

\begin{proposition}[Lipschitz Continuity]
\label{prop:lipschitz_continuity}
For $2 < p < \infty$, if $f \in L^{1+p/D}(\mathbb{R}^D)$, then $RDAD$ is $[\frac{1}{m^{1/p}} \|f\|_{L^{1+p/D}}^{1/p+1/D}]$-Lipschitz continuous. If $f$ is bounded, then both $DAD$ and $RDAD$ are $\|f\|_\infty^{1/D}$-Lipschitz.
\end{proposition}

Finally we present a statistical convergence result that extends Theorem 9 of \cite{chazal18_DTM_statistics}.

\begin{proposition}[Functional Normality]\label{prop:functional_consistency}
Suppose $X_1, ..., X_N$ is an independent and identically distributed sample with a density $f$ that is continuous and has a compact support. If $\frac{k_\text{den}}{\log N} \to \infty$, $\frac{k_\text{den}}{N} \to 0$ and $\frac{k_\text{DTM}}{N} \to m$, then on every compact set $K$ in $\mathbb{R}^D$, $\sqrt{N} (\widehat{RDAD}^2 - RDAD^2)$ converges weakly in $L^\infty(K)$ to a centered Gaussian process with covariance kernel
$$\kappa(x, y) = \frac{1}{m^2} \int_0^{(F_x^{-1}(m))^2}\int_0^{(F_y^{-1}(m))^2} \left(P(E_{x, t} \cap E_{y, s}) - F_x(\sqrt{t}) F_y(\sqrt{s})\right) ds dt,$$
where $E_{z, r} = \{\xi: f(\xi)^{1/D} d(z, \xi) \leq \sqrt{r}\}$ for every $z \in \mathbb{R}^D$ and $r > 0$.
\end{proposition}

\section{Bootstrapping}
\label{sec:bootstrapping}

\begin{figure}[h]
\centering

\includegraphics[width = 0.45\linewidth]{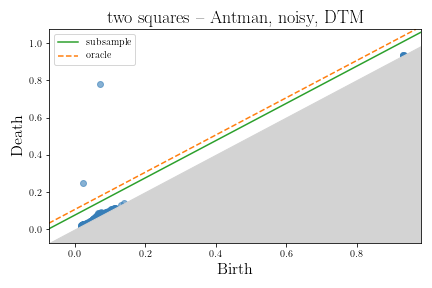}
\includegraphics[width = 0.45\linewidth]{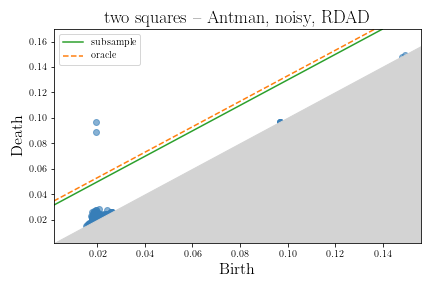}

\caption{Dimension-1 persistence diagrams of different filtration functions for the additive noise-corrupted ``Antman" two-square dataset with confidence bands. Blue points are points in the dimension-1 empirical persistence diagram. The green solid lines and the orange dashed lines are the confidence bands constructed by subsample and oracle bootstrapping respectively.}
\label{fig:bisquare2D_noisy_persistenceBootstrapped}
\end{figure}

In order to distinguish statistically significant topological signals from noise, a confidence band like those in \cite{fasy14_confidence_set,chazal18_DTM_statistics} is desirable. Below, we discuss how to construct such a confidence band, and compare the resultant band with what an ``oracle", who knows the density $f$ from which the points are sampled, would construct. We show that the two bands are empirically similar in the additive noise-corrupted ``Antman'' example. A more complicated example will be considered in \cref{sec:voronoi}.

Let $\hat P$ be the empirical persistence diagram. The confidence set we construct is a bottleneck metric ball centered at $\hat P$.
Even though an ``oracle" who knows the density $f$ would, presumably, know the significant features, we still imagine that they will simply determine an appropriate radius for the ball around the empirical persistence diagram by sampling a certain number of independent samples of the same size from $f$, and approximating, by the empirical quantile, the true $(1-\alpha)$-quantile\footnote{This $\alpha$ denotes the level of significance and is different from the parameter $\alpha$ for moderate tails. The meaning of $\alpha$ should be clear from context.} of the bottleneck distance from $\hat P$ to a persistence diagram from an independent sample from $f$. If the original sample has been corrupted, the oracle samples are corrupted by the same mechanism too.

The density $f$ is not known to a ``non-oracle". In this case, we adapt the subsampling method proposed in \cite{chazal18_DTM_statistics}.
Specifically, we generate $B$ samples ($B = 100$ in our implementation) of $N$ random vectors drawn from $X_1, ..., X_N$ with replacement and compute $B$ persistent diagrams $P^*_1, ..., P^*_B$. The radius of the confidence set is the empirical $(1-\alpha)$-quantile of the bottleneck distances of $\hat P$ from the $P^*_i$'s.

In \cite{fasy14_confidence_set}, each bootstrap sample contains $o(N)$ points. However, since the scale of the proposed filtration changes with the sample size, we fix the bootstrap sample size to be $N$ to ensure comparability of the bootstrap sample and the empirical sample.

Consider, again, the additive noise-corrupted ``Antman'' two-square dataset on the right of \cref{fig:bisquare2D_noisy_outlier_samplePoints}. The persistence diagrams for the distance-to-measure filtration and the RDAD filtration are shown in \cref{fig:bisquare2D_noisy_persistenceBootstrapped} with different confidence bands. Note that in both figures the bands constructed by oracle bootstrapping and by subsample bootstrapping are very close to each other, and both of them identify correctly the two true loops.

\section{Simulations} 
\label{sec:simulations}

We illustrate the utility of the proposed filtration on synthetic and real data including a brief description of the computation of the persistent diagrams of the empirical $\widehat{RDAD}$ filtration.

First, in \cref{sec:voronoi}, motivated by the multiscale nature of the cosmic web, we attempt to identify Voronoi cells with a range of sizes from observed sample points on cell edges.

Then, in \cref{sec:cellularTowers}
we attempt to apply our method to a dataset of cellular towers in the United States to discover mesoscale holes formed by geography and missing data.

Model parameters and sample sizes of different datasets are summarized in the tables in \cref{sec:simulation_parameters}.

\subsection{Implementation}
\label{sec:simulations_implementation}

We approximate the empirical $\widehat{RDAD}$ function by a function that is piecewise constant on a fine grid and coincides with the $\widehat{RDAD}$ function at the center of each grid cell. This produces a cubical filtration. We use the implementation in \cite{gudhi_CubicalComplex}. This computation is feasible for 2- to 3-dimensional data, but we confine ourselves to 2-dimensional data for easier visualization. A byproduct of computing with cubical filtrations is that we can locate the pixel at which each codimension-1 hole is eventually filled.

All homological computations are done with coefficients in $\mathbb{Z}/{11 \mathbb{Z}}$, which is the default field in \cite{gudhi_CubicalComplex}.

\subsection{Recovery of Synthetic Voronoi Cells}
\label{sec:voronoi}

\begin{figure}[h]
\centering
\includegraphics[width = 0.9\linewidth]{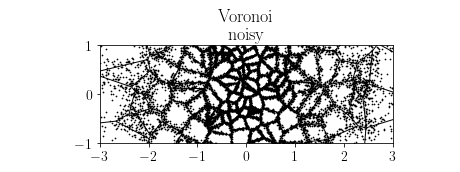}

\caption{Sample points of the noisy Voronoi dataset.}
\label{fig:voronoi_noisy_samplePoints}
\end{figure}

\begin{figure}[h]
\centering
\includegraphics[width = 0.9\linewidth]{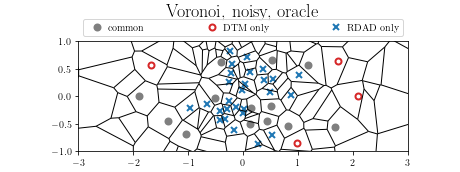}

\includegraphics[width = 0.9\linewidth]{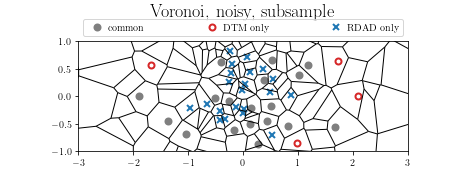}

\caption{Significant loops under different filtrations and different bootstrapping methods for the noisy Voronoi dataset. Significant loops under different filtrations but the same bootstrapping methods share the same plot and are distinguished by their colors and markers.}
\label{fig:voronoi_noisy_loopCmp}
\end{figure}

\begin{figure}[h]
\centering
\includegraphics[width = 0.45\linewidth]{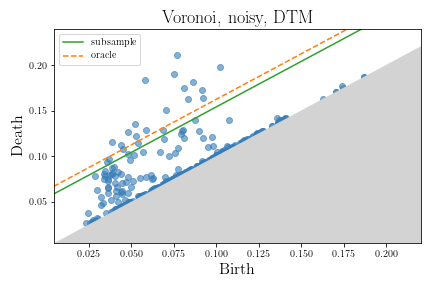}
\includegraphics[width = 0.45\linewidth]{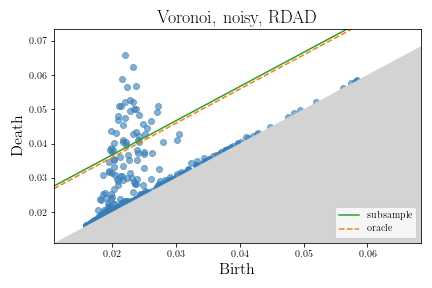}

\caption{Dimension-1 persistence diagrams of different filtration functions for the noisy Voronoi dataset with confidence bands. Blue points are points in the dimension-1 empirical persistence diagram. The green solid lines and the orange dashed lines are the confidence bands constructed by subsample and oracle bootstrapping respectively.}
\label{fig:voronoi_noisy_persistence}
\end{figure}

In this subsection, we attempt to recover the Voronoi cells with the proposed filtration from a sample of points near edges of a planar Voronoi diagram with cells of different sizes and different densities on their edges. This is the same dataset as the one on the left of \cref{fig:TDA_illustration_small_features}, and it is motivated by the cosmological Voronoi model in \cite{icke91_cosmic_void_voronoi} (see also \cite{pranav16_cosmic_void_TDA, xu19_cosmic_void_TDA}), where galactic matter is concentrated on walls and filaments of a Voronoi diagram, whose cell sizes span a wide range of scales, as observed in \cite{aragoncalvo12_cosmic_void_hierarchy, wilding21_cosmic_void_TDA}. 
For easier visualization, we consider only planar Voronoi diagrams.

We experiment with Voronoi diagrams in which the cells in the center of the diagram tend to be smaller. A point is sampled by first choosing a random cell and then choosing a uniform point on its boundary. This results in a higher sampling density on boundaries of smaller cells. We further inject additive noise. Further details of the data generation process may be found in \cref{sec:simulation_parameters}.

We compare the performances of the proposed filtration against that of the distance-to-measure filtration. The sample points are shown in \cref{fig:voronoi_noisy_samplePoints}, the persistence diagrams are shown in \cref{fig:voronoi_noisy_persistence}, and the significant loops found by oracle and subsample bootstrapping are shown in \cref{fig:voronoi_noisy_loopCmp}.

As shown in \cref{fig:voronoi_noisy_loopCmp}, while the proposed method misses some of the bigger cells detected by distance-to-mesaure, with the sizes of different loops normalized by density, it detects many smaller cells in the middle that distance-to-measure cannot detect.

We also note the closeness of the confidence bands of oracle and subsampling bootstrapping in \cref{fig:voronoi_noisy_persistence}. This serves as empirical evidence that subsampling bootstrapping
does not suffer heavy damage from not being able to generate new sample from the true density.

\subsection{Real Data}
\label{sec:cellularTowers}

\begin{figure}[h]
\centering
\includegraphics[width = 0.45\linewidth]{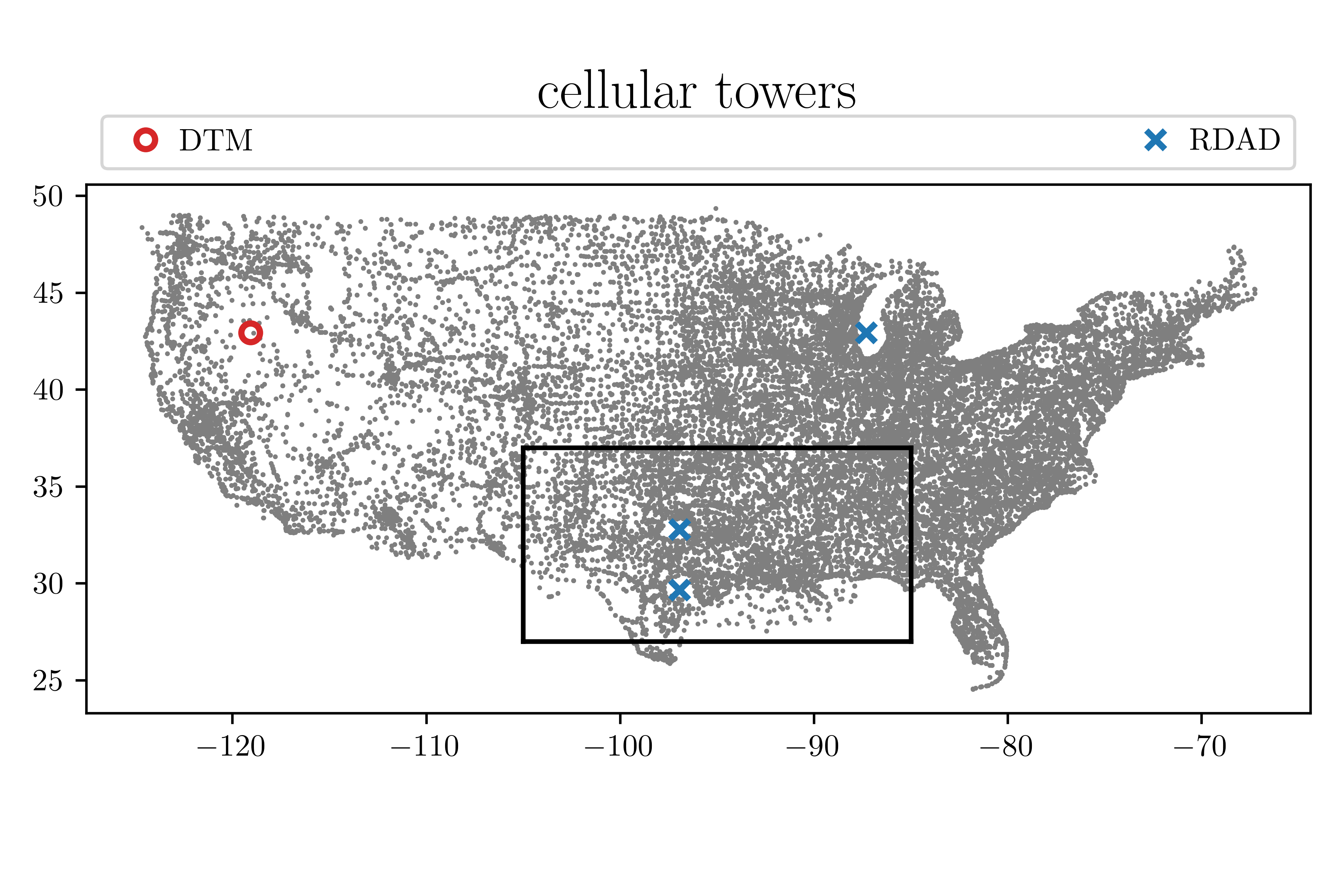}
\includegraphics[width = 0.45\linewidth]{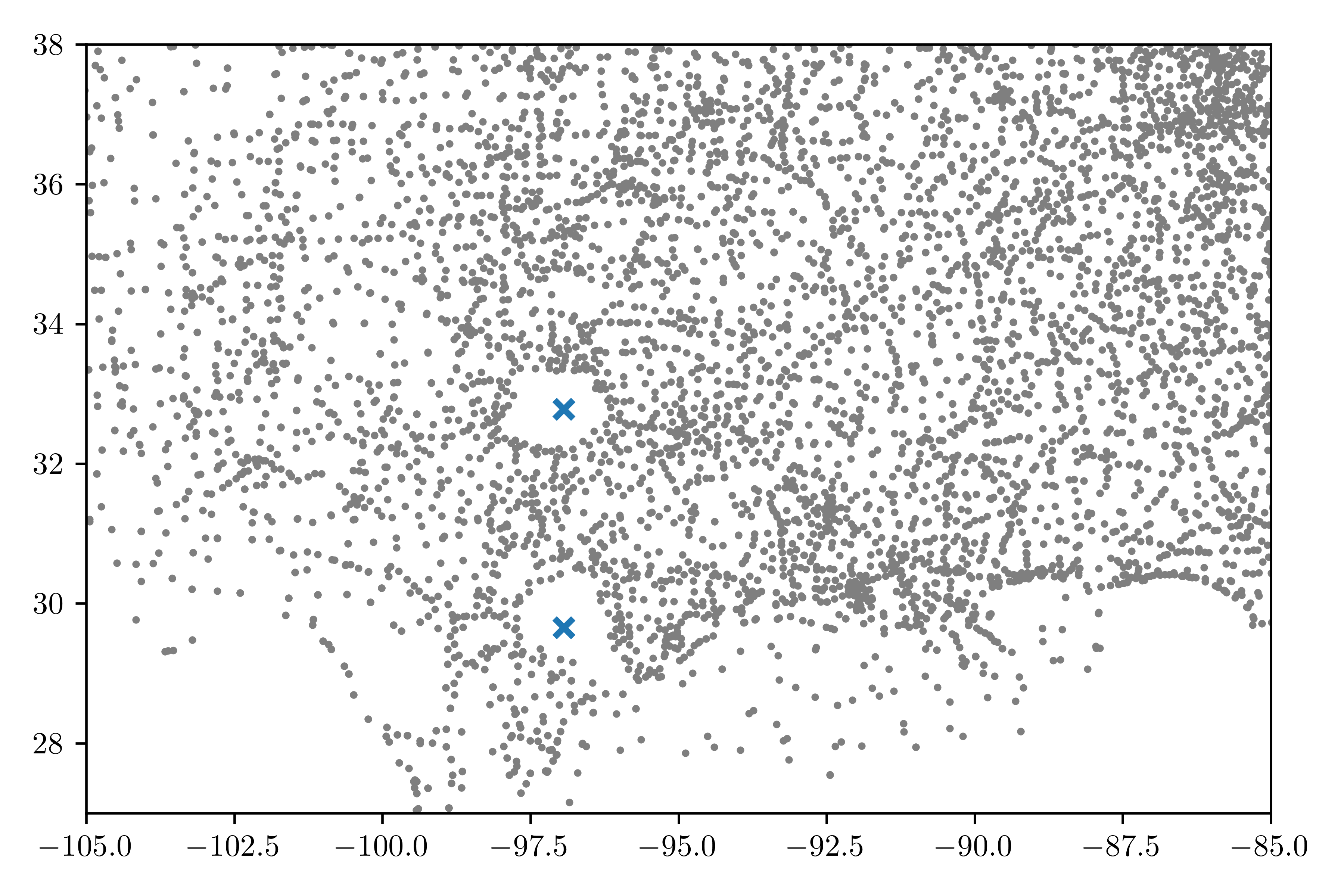}

\caption{Sample points of the cellular tower dataset (grey dots) and significant loops under subsample bootstrapping for different filtrations (red hollow circles and blue crosses). The black rectangle, which contains two holes detected by RDAD, is blown up and shown on the right subplot.}
\label{fig:cellularTowers_clean_loopCmp}
\end{figure}

\begin{figure}[h]
\centering
\includegraphics[width = 0.45\linewidth]{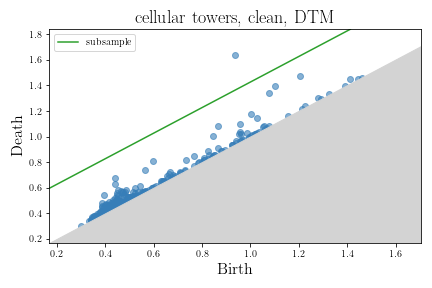}
\includegraphics[width = 0.45\linewidth]{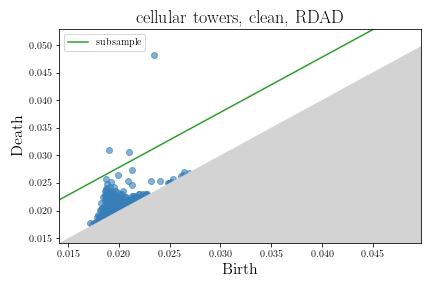}

\caption{Dimension-1 persistence diagrams of different filtration functions for the cellular tower dataset with confidence bands constructed by subsample bootstrapping.}
\label{fig:cellularTowers_clean_persistence}
\end{figure}

We also apply our method to real data. The distance-to-measure filtration and the RDAD filtration are applied to an open dataset \cite{HIFLD21_cellurlar_towers} of cellular tower locations recorded by the Federal Communications Commission (FCC). The two filtrations reveal uninhabited regions in the United States and regions of missing data. As expected, the regions found by the distance-to-measure filtration are large while the small ones are detected only by the RDAD filtration.
Details of the dataset and our preprocessing method are summarized in \cref{sec:simulation_parameters}.

The scatter plot of the cell towers is shown in 
\cref{fig:cellularTowers_clean_loopCmp}, followed by the persistence diagrams of the two filtrations in \cref{fig:cellularTowers_clean_persistence}. Even without the aid of the confidence bands, one point is conspicuously far away from the diagonal in the persistence diagram of each filtration. The RDAD filtration picks up 2 more significant loops.

The two filtrations pick up completely different homology classes. The class picked up by the distance-to-measure filtration is near Steens Mountain Wilderness in Oregan. The 3 classes picked up by the RDAD filtration are Lake Michigan; Dallas, Texas; and the Texan region surrounded by Houston, Austin and San Antonio. The last two regions have considerable population, and the sparsity of cellular towers there is likely due to the dataset's incompleteness.

The homology class picked up by the distance-to-measure filtration is a large sparsely populated area with few cellular towers if any. Those picked up by the RDAD filtration are comparatively smaller regions with an abrupt drop in density. The distance-to-measure filtration fails to pick up the smaller homology classes. Even Lake Michigan is too small because of its narrowness.
The RDAD filtration cannot detect the large sparsely populated regions because the drop in density there is not sharp enough -- nearby regions have very low density anyway.

\section{Discussion}
\label{sec:discussion}

While the proposed filtration does prolong the persistences of homology classes of high-density regions in a robust manner, a range of practical, theoretical, statistical and computational issues warrant further investigation.

In practice, while the proposed filtration is designed to handle data with non-uniform density, and can detect well-separated features in a scale-invariant manner, certain large low-density region may cause difficulties.

The effects of the parameters $k_\text{DTM}$ and $k_\text{den}$ need to be studied in greater detail. For $k_\text{den}$, one may use cross-validation to choose a reasonable $k_\text{den}$. The effect of different choices of $k_\text{DTM}$ is much less understood. In particular, cross-validation does not directly apply because topological signals are inherently global in nature. 

Statistically, a theoretical foundation of the bootstrapping method remains to be developed. This is challenging even in the case of the simpler distance-to-measure filtration.

Computationally, in order to obtain a reasonable level of precision, the calculations above are done on a grid, at the cost of restricting the ambient dimension. Computation of the associated Rips complex is likely more feasible.

\section{Conclusion}
\label{sec:conclusion}

The novel Robust Density-Aware Distance filtration is proposed in the present work for studying data with a non-uniform density. It is designed to make small holes of high-density regions more prominent. It is scale-invariant, and the persistences of homology classes in the proposed filtration depend on the shapes rather than the sizes of the features. Robustness against noise is enhanced through the incorporation of the idea of distance-to-measure. A bootstrapping method
is proposed to gauge the significance of a topological feature.
The properties of the proposed filtration have been established both theoretically and empirically with artificial and real datasets.

\backmatter

\bmhead{Acknowledgements}

The authors would like to thank Andrew Thomas and Jason Manning for insightful conversations.

This research was conducted with support from the Cornell University Center for Advanced Computing, which receives funding from Cornell University, the National Science Foundation, and members of its Partner Program.

All topological computations were done with Gudhi \cite{gudhi_CubicalComplex}. Nearest-neighbor computations were done with scikit-learn \cite{scikit-learn}. 
Numerical computations were done with Numpy \cite{numpy}, Scipy \cite{scipy} and Pandas \cite{pandas_software, pandas_article}.  Codes were compiled with Numba \cite{lam15_numba}. Graphs were generated with Matplotlib \cite{matplotlib}.

\bmhead{Funding}

Chunyin Siu was supported by Croucher Scholarship for Doctoral Studies.
Gennady Samorodnitsky was supported by the NSF grant DMS-2015242 and AFOSR grant FA9550-22-1-0091.
Christina Lee Yu was supported by an Intel Rising Stars Award and NSF grants CCF-1948256 and CNS-1955997.

The authors have no relevant financial or non-financial interests to disclose.

\begin{appendices}

\section{Proofs}
\label{sec:proofs}
We first establish the main results in \cref{sec:properties_small_dense_holes,sec:properties_scale_invariance,sec:properties_robustness} assuming the results from \cref{sec:properties_further}. The latter results are proven in \cref{sec:proofs_further_properties}. Lemmas will be needed along the way. Unless they are used repeatedly, they will be introduced and proven after they are first used.

\subsection{Proofs of \cref{thm:RDAD_interleaving} and its Corollaries}

We first prove the theorem.

\begin{proof}[Proof of \cref{thm:RDAD_interleaving}]
Suppose first that $x \in \Omega_i$ for some $i$. In this case, the lower bound is trivial. Since $P(B(x, \rho_m)) \geq m$, there exists a mass-$m$ measure $\nu_x$ subordinate to $P$ (in the sense defined in \cref{lem:variational_characterization}) whose support lies in $B(x, \rho_m)$.
By the variational characterization of RDAD (\cref{lem:variational_characterization}),
\begin{align*}
RDAD(x)
&\leq \sqrt{\frac{1}{m} \int [f(\xi)^{1/D} d(\xi, x)]^2 d\nu_x(\xi)}
\leq \|f\|_\infty^{1/D} \rho_m
= O(\rho_m).
\end{align*}
This gives the upper bound in the theorem.

Let now $y \in \mathbb{R}^D \setminus \cup \Omega_i$, and for each $i$, let $x_i$ be the point nearest to $y$ in $\Omega_i$. Since $x_i \in \partial \Omega_i$, $f(x_i) = t_i$ by assumption.

We first show the upper bound. Again, for each $i$, let $\nu_i$ be a mass-$m$ measure subordinate to $P$ whose support lies in $B(x_i, \rho_m) \cap f^{-1}[0, a t_i]$. Then
\begin{align*}
RDAD(y)
&\leq \sqrt{\frac{1}{m} \int [f(\xi)^{1/D} d(\xi, x)]^2 d\nu_i(\xi)}
\\& \leq (a t_i)^{1/D}  (d(y, x_i) + \rho_m)
\\& = a^{1/D} t_i^{1/D} d(y, x_i) + O(\rho_m).
\end{align*}
Since this is true for every $i$, the upper bound follows.

For the lower bound, let $r$ be such that $F_y(r) > m/2$. Since
$P(\mathbb{R}^D \setminus \cup \Omega_i) \leq m/2$, the set
$$\{\xi: f(\xi)^{1/D} d(\xi, y) \leq r\}$$
must contain a point $\xi_0 \in \Omega_i$ for some $i$. Then $$r \geq f(\xi_0)^{1/D} d(\xi_0, y) \geq t_i^{1/D}d(x_i, y) = g_i(y) \geq g(y).$$
Therefore, for $q > m/2$, since $F_y^(F_y^{-1}(q)) \geq q > m/2$, we have $F_y^{-1}(q) \geq g(y)$. Hence
$$RDAD(y) \geq \sqrt{\frac{1}{m} \int_{m/2}^m F_y^{-1}(q)^2 dq} \geq \frac{1}{\sqrt{2}} g(y).$$
\end{proof}

\begin{proof}[Proof of \cref{cor:RDAD_persistence_lower_bound}]

Under the sublevel filtration of $g_i$, the relevant homology class has birth and death times $t_i^{1/D} \beta$ and $t_i^{1/D} \delta$. By \cref{lem:separation_Omega} below, it also has the same birth and death time under $g$. 

Let $\tilde \beta = (at_i)^{1/D} \beta + O(\rho_m)$ and $\tilde \delta = \frac{1}{\sqrt{2}} t_i^{1/D} \delta$. \cref{thm:RDAD_interleaving} implies
$$g^{-1}[0, t_i^{1/D} \beta] \subseteq RDAD^{-1}[0, \tilde \beta] \subseteq RDAD^{-1}[0, \tilde \delta] \subseteq g^{-1}[0, t_i^{1/D} \delta],$$
the class's persistence under RDAD is then at least $\tilde \delta - \tilde \beta$.

\end{proof}

\begin{lemma}\label{lem:separation_Omega}
The homology persistence module of the sublevel filtration of $g_i$ at levels strictly less than $s_i$ is a summand of that of $g$.
\end{lemma}

\begin{proof}
We claim that for every $\varepsilon > 0$, on $g_i^{-1}[0, s_i - \varepsilon]$, $g_j > s_i > g_i$ for every $j \neq i$. Assuming this claim, $g_i^{-1}[0, s_i - \varepsilon]$ is then a component of $g^{-1}[0, s_i - \varepsilon]$, and on this component, $g_i$ and $g$ coincide. The lemma then follows.

We now establish the claim. If $g_i(x) < s_i$, then for every $j \neq i$,
$$t_i^{1/D}d(x, \Omega_i) < s_i \leq \frac{d(\Omega_i, \Omega_j)}{t_i^{-1/D} + t_j^{-1/D}},$$
and hence
\begin{align*}
g_j(x)
&= t_j^{1/D} d(x, \Omega_j)
\\& \geq t_j^{1/D} (d(\Omega_i, \Omega_j) - d(x, \Omega_i))
\\& \geq t_j^{1/D} [(t_i^{-1/D} + t_j^{-1/D}) s_i - d(x, \Omega_i)]
\\&= s_i + t_j^{1/D} t_i^{-1/D}(s_i - t_i^{1/D} d(x, \Omega_i))
\\&> s_i.
\end{align*}
The claim then follows.
\end{proof}

Since \cref{cor:codim_1_features} follows directly from applying Alexander duality to \cref{cor:RDAD_persistence_lower_bound}, we omit the proof.

\subsection{Proof of \cref{prop:scale_invariance}}

Fix $x$ and let $\tilde x = ax + b$. Let $\tilde P$ be the measure induced by $\tilde f$ and $\tilde F_{\tilde x}(r) = \tilde P[\tilde f(\tilde X)^{1/D} d(\tilde X, \tilde x) < r]$.

It suffices to show $\tilde F_{\tilde x}(r) = F_x(r)$ for every $r$.

Let $E_x = \{y: f(y)^{1/D} d(y, x) \leq r\}$ and $\tilde E_{\tilde x} = \{\tilde y: \tilde f(\tilde y)^{1/D} d(\tilde y, \tilde x) \leq r\}$. Then it can be readily verified that $\tilde E_{\tilde x}  = aE_x + b = \{ay + b: y \in E_x\}$. Then change of variable formula implies
\begin{align*}
\tilde F_{\tilde x}(r)
= \int_{\tilde E_{\tilde x}} \tilde f(\tilde y) d\tilde y
= \int_{E_x} f(y) dy
= F_x(r),
\end{align*}
from which the equality of the two functions follows.

The equality of the persistence diagrams is now apparent, because the sublevel sets of the two filtrations are scaled and translated versions of each other, and hence share the same topological features.

For the empirical version, the claim is clear because the scaling multiplies both numerator and denominator of $d(x, X_{(i)})/d_{(i)}$ by the same factor, and the translation affects neither of them. The equality of persistence diagrams is analogous to the population case.

\subsection{Proofs of \cref{thm:robust} and its Corollary}

We first establish some lemmas which will be used throughout the proof of \cref{thm:robust}. 

The first lemma is a corollary of H\"{o}lder's inequality, and we still call it H\"{o}lder's inequality.

\begin{lemma}[H\"{o}lder's Inequality]\label{lem:Holder_submeasure}
Let $\mu$ and $P$ be measures on $X$. Suppose $\mu$ is subordinate to $P$ (in the sense stated in \cref{lem:variational_characterization}). Then for $p, q, r \in [1, \infty]$ satisfying $\frac{1}{r} = \frac{1}{p} + \frac{1}{q}$,
$$\|\varphi\|_{L^r(\mu)} \leq \mu(X)^{1/q}\|\varphi\|_{L^p(P)}.$$
\end{lemma}

\begin{proof}
Since $\mu$ is subordinate to $P$, its Radon-Nikodym derivative $\frac{\partial \mu}{\partial P}$ is bounded above by $1$. Since $q/r = q/p + 1 \geq 1$, $\left ( \frac{\partial \mu}{\partial P}\right )^{q/r} \leq \frac{\partial \mu}{\partial P}$.
H\"{o}lder's inequality implies
$$\|\varphi\|_{L^r(\mu)} = \left\|\varphi \left ( \frac{\partial \mu}{\partial P}\right )^{1/r} \right \|_{L^r(P)} \leq \|\varphi\|_{L^p(P)} \left \|\left(\frac{\partial \mu}{\partial P}\right)^{1/r}\right\|_{L^q(P)},$$
and
$$\left\|\left(\frac{\partial \mu}{\partial P}\right)^{1/r}\right\|_{L^q(P)} = \left[\int \left(\frac{\partial \mu}{\partial P}\right)^{q/r} dP \right]^{1/q} \leq \left (\int \frac{\partial \mu}{\partial P} dP \right)^{1/q} = \mu(X)^{1/q}.$$
The result then follows.
\end{proof}

We will also need the following estimate repeatedly.

\begin{lemma}[Estimate on the Distance from a Compact Set]\label{lem:rho_estimate}
Let $1 \leq p < \infty$. Let $K$ be a compact set in $\mathbb{R}^D$ and $x_0 \in K$. Let $\rho(x) = d(x, x_0) = \lvert x - x_0 \rvert$. For any probability measure $P$ with a finite $p^{th}$ moment, $\|\rho\|_{L^p(P)}$ has an upper bound that depends only on $p, K$ and the $p^{th}$ moment of $P$ (but not on $x_0$).
\end{lemma}

\begin{proof}
$$\|\rho\|_{L^p(P)} = \left(\int \abs{x - x_0}^p dP\right)^{1/p} \leq \abs{x_0} + \left(\int \abs{x}^p dP\right)^{1/p} \leq \max_{x' \in K} \abs{x'} + \left(\int \abs{x}^p dP\right)^{1/p}.$$
\end{proof}

We will need other lemmas, but since they will be used only once, we establish them after proving the theorem.

\begin{proof}[Proof of \cref{thm:robust}]

Let $P$ and $\tilde P$ be the measures induced by $f$ and $\tilde f$ on $\mathbb{R}^D$.

Fix $x_0 \in K$ and let $\rho(x) = d(x, x_0) = \lvert x - x_0 \rvert$. Let $s = (2^{-1} - r^{-1})^{-1}$.

By the variational characterization of RDAD (\cref{lem:variational_characterization}), let $\nu^*$ be a mass-$m$ measure subordinate to $P$ such that
$$\sqrt{m} RDAD(x_0; f, m) = \|f^{1/D} \rho\|_{L^2(\nu^*)}.$$
Let $\Pi$ be an optimal coupling between $P$ and $\tilde P$ such that
$$W_p(f, \tilde f) = \| \lvert x - y \rvert\|_{L^p(\Pi)}.$$
Let $\tilde \nu$ and $\pi$ be subordinate to $\tilde P$ and $\Pi$ respectively such that $\pi$ is a coupling between $\nu^*$ and $\tilde \nu$. They exist by \cref{lem:transport_submeasure} below.

Again, by the variational characterization of RDAD (\cref{lem:variational_characterization}) and triangle inequality,
\begin{align}
&\sqrt{m} RDAD(x_0; \tilde f, m) \notag
\\& \leq \|\tilde f^{1/D} \rho\|_{L^2(\tilde \nu)}\notag
\\& \leq \|f^{1/D} \rho\|_{L^2(\nu^*)} + \left \lvert \|f^{1/D} \rho\|_{L^2(\nu^*)} - \|f^{1/D} \rho\|_{L^2(\tilde \nu)} \right \rvert + \|(\tilde f^{1/D} - f^{1/D}) \rho\|_{L^2(\tilde \nu)} \notag
\\&= \sqrt{m} RDAD(x_0; f, m) + \left \lvert \|f^{1/D} \rho\|_{L^2(\nu^*)} - \|f^{1/D} \rho\|_{L^2(\tilde \nu)} \right \rvert + \|(\tilde f^{1/D} - f^{1/D}) \rho\|_{L^2(\tilde \nu)}.
 \label{eqn:RDAD_robustness_main_estimate}
\end{align}

We first estimate the first error term $A = \left \lvert \|f^{1/D} \rho\|_{L^2(\nu^*)} - \|f^{1/D} \rho\|_{L^2(\tilde \nu)} \right \rvert$.
Since $\pi$ is a coupling between $\nu^*$ and $\tilde \nu$, we have
\begin{align*}
\left \lvert \|f^{1/D} \rho\|_{L^2(\nu^*)} - \|f^{1/D} \rho\|_{L^2(\tilde \nu)} \right \rvert
&= \left \lvert \|(f^{1/D} \rho)(x)\|_{L^2(\pi)} - \|(f^{1/D} \rho)(y)\|_{L^2(\pi)} \right \rvert
\\&\leq \|(f^{1/D} \rho)(x)-(f^{1/D} \rho)(y)\|_{L^2(\pi)}.
\end{align*}

Since $\rho$ and $f^{1/D}$ are Lipschitz and $f$ is bounded,
$$ \lvert (f^{1/D}\rho)(x) - (f^{1/D}\rho)(y) \rvert \leq \|f\|_\infty^{1/D} \lvert x - y \rvert + \text{Lip}(f^{1/D}) \rho(x) \lvert x-y \rvert = O((1 + \rho(x))\lvert x-y \rvert).$$

Let $u = (r^{-1} - p^{-1})^{-1}$. Recallng the definition of $s$ at the beginning of the proof, we have $2^{-1} = p^{-1} + u^{-1} + s^{-1}$, and hence H\"{o}lder's inequality gives
\begin{align*}
\|(1 + \rho(x))\lvert x - y \rvert\|_{L^2(\pi)}
&\leq m^{1/s} \|1 + \rho(x)\|_{L^u(\Pi)} \|\lvert x - y \rvert\|_{L^p(\Pi)}
\\&= m^{1/s} \|1 + \rho\|_{L^u(P)} W_p(f, \tilde f).
\end{align*}
By \cref{lem:rho_estimate}, the central factor of the last line is bounded. Therefore,
$$A = O(m^{1/s} W_p(f, \tilde f)).$$

For the second error term
$$B = \|(\tilde f^{1/D} - f^{1/D}) \rho\|_{L^2(\tilde \nu)}$$
in \cref{eqn:RDAD_robustness_main_estimate}, H\"older's inequality and 
\cref{lem:difference_roots_combined}
below (applied to $\mu = \tilde P$, $\varphi = \tilde f$, $\psi = f$, $\theta = \rho$, $p = \infty$)
gives
$$B \leq m^{1/s}\|\tilde f^{1/D} - f^{1/D} \rho\|_{L^2(\tilde P)} = O(m^{1/s}\|f - \tilde f\|_\infty^{\frac{D+r}{(D+1)(r+1)}}).$$

Combining the estimates of $A$ and $B$, one side of the bound then follows. The other side is analogous, but not completely symmetric, because $\tilde f^{1/D}$ is \emph{not} assumed to be Lipschitz. We modify the argument above as follows.

Again, let $\sqrt{m}RDAD(x_0; \tilde f) = \|\tilde f^{1/D} \rho\|_{L^2(\tilde \nu^*)}$. Let $\nu$ and $\tilde \pi$ be subordinate to $P$ and $\Pi$ such that $\tilde \pi$ is a coupling between $\nu$ to $\tilde \nu^*$. Then
\begin{align*}
&\sqrt{m} RDAD(x_0; f) \notag
\\& \leq \|f^{1/D} \rho\|_{L^2(\nu)}\notag
\\& \leq \|\tilde f^{1/D} \rho\|_{L^2(\tilde \nu^*)} + \|(\tilde f^{1/D} - f^{1/D}) \rho\|_{L^2(\tilde \nu^*)} + \left \lvert \|f^{1/D} \rho\|_{L^2(\tilde \nu^*)} - \|f^{1/D} \rho\|_{L^2(\nu)} \right \rvert \notag
\\&= \sqrt{m}RDAD(x_0; \tilde f) + \|(\tilde f^{1/D} - f^{1/D}) \rho\|_{L^2(\tilde \nu^*)} + \left \lvert \|f^{1/D} \rho\|_{L^2(\tilde \nu^*)} - \|f^{1/D} \rho\|_{L^2(\nu)} \right \rvert.
\end{align*}
Since $\tilde f$ only appears in the last line in the difference $\tilde f^{1/D} - f^{1/D}$, the analysis of these error terms are now completely analogous to that of $A$ and $B$. The proof then follows.
\end{proof}

\begin{lemma}\label{lem:transport_submeasure}
Let $P$, $\tilde P$ and $\mu$ be measures on a common measurable space, and $\pi$ be a coupling between $P$ and $\tilde P$. If $\mu$ is subordinate to $P$ (in the sense defined in \cref{lem:variational_characterization}), then there exist measures $\tilde \mu, \tilde \pi$ subordinate to $\tilde P$ and $\pi$ respectively such that $\tilde \mu$ has the same mass as $\mu$ and $\tilde \pi$ is a coupling between $\mu$ and $\tilde \mu$.
\end{lemma}

\begin{proof}
Since $\mu$ is subordinate to $P$, its Radon-Nikodym derivative $\frac{\partial \mu}{\partial P}$ is bounded above by $1$.
Let $d\tilde \pi(x, y) = \frac{\partial \mu}{\partial P}(x) d\pi(x, y)$. Then $\tilde \pi$ is subordinate to $\pi$; and its $y$-marginal, which we call $\tilde \mu$, is subordinate to $\tilde P$ and has the same mass as $\mu$.
\end{proof}

\begin{lemma}\label{lem:difference_roots_combined}
Let $\varphi$ be a bounded moderately tailed density on $\mathbb{R}^D$ with parameters $C, \alpha$, where $\alpha \geq D^2 - D$. Let $\mu$ be the measure induced by $\varphi$ on $\mathbb{R}^D$. Let $\psi$ and $\theta$ be measurable functions on $\mathbb{R}^D$. Suppose $\psi$ is bounded and nonnegative, and $\theta \in L^u(\mu)$ for every $1 \leq u < \infty$. Then for $1 < r < p \leq \infty$, if $\|\varphi - \psi\|_{L^p(\mu)} \leq 1$, then
$$\|(\varphi^{1/D} - \psi^{1/D})\theta\|_{L^r(\mu)} = O(\|\varphi - \psi\|_{L^p(\mu)}^
{\frac{D+r}{(D+1)(r+1)}}
),$$
where the big-Oh constant depends only on $\|\varphi\|_\infty, \|\psi\|_\infty, p, r, D, C$ and the $L^u(\mu)$ norms of $\theta$, where $u$ ranges over $[r, \infty)$.
\end{lemma}

\begin{proof}

The claim is trivial for $D = 1$. Below we assume $D > 1$.

Let $0 < \eta \leq 1$. We will choose its value later. Note that the following sets cover the whole space:
$$E = \{\varphi > \eta\}, \quad F = \{\lvert \varphi - \psi \rvert > \eta^{1-1/D}\}, \quad G = \{\varphi \leq \eta, \psi \leq 2\eta^{1-1/D}\}.$$
We estimate the $L^r$-norm of $(\varphi^{1/D} - \psi^{1/D})\theta$ on each of the sets above.
Let $u = (r^{-1} - p^{-1})^{-1} \geq r$. 

For $E$, H\"{o}lder's inequality and the elementary inequality $\lvert a^{1/D} - b^{1/D} \rvert \leq \frac{\lvert a - b \rvert}{a^{1-1/D}}$ give 
\begin{align*}
\|(\varphi^{1/D} - \psi^{1/D})\theta \mathbf{1}_E\|_r
&\leq \|\theta\|_u \|(\varphi^{1/D} - \psi^{1/D})\mathbf{1}_E\|_p
\\&\leq \frac{1}{\eta^{1-1/D}} \|\theta\|_u \|\varphi - \psi\|_p
\end{align*}

For $F$, H\"{o}lder's inequality and Markov's inequality imply
\begin{align*}
\|(\varphi^{1/D} - \psi^{1/D})\theta \mathbf{1}_F\|_r
&\leq (\|\varphi\|_\infty^{1/D} + \|\psi\|_\infty^{1/D}) \|\theta\|_u \|\mathbf{1}_F \|_p
\\&\leq (\|\varphi\|_\infty^{1/D} + \|\psi\|_\infty^{1/D}) \|\theta\|_u \left(\frac{1}{\eta^{1-1/D}} \|\varphi - \psi\|_p\right)
\end{align*}

For $G$, let $q \in (r, \infty)$, whose value will be chosen later, and let $v = (r^{-1} - q^{-1})^{-1} \geq r$. H\"{o}lder's inequality then gives
\begin{align*}
\|(\varphi^{1/D} - \psi^{1/D})\theta \mathbf{1}_G\|_r
&\leq (1 + 2^{1/D}) \eta^{(1-1/D)(1/D)} \|\theta\|_v \|\mathbf{1}_{\varphi \leq \eta} \|_q
\\&\leq (1 + 2^{1/D}) \eta^{(1-1/D)(1/D)} \|\theta\|_v \|(\mu\{\varphi \leq \eta\}^{1/q}).
\end{align*}

For the last factor,
\begin{align*}
\mu\{\varphi \leq \eta\}
= \int_{\varphi \leq \eta} \varphi(x) dx
\leq \int_{\substack{\varphi \leq \eta \\ \lvert x \rvert \leq R}} \varphi(x) dx +  \int_{\lvert x \rvert > R} \varphi(x) dx
\leq \omega_D \eta R^D + CR^{-\alpha}.
\end{align*}
Since $\alpha \geq D^2 - D$, putting $R = \eta^{-1/D^2}$ gives
$$\mu\{\varphi \leq \eta\}^{1/q} = O(\eta^{(1-1/D)(1/q)}),$$
and hence
$$\|(\varphi^{1/D} - \psi^{1/D})\theta \mathbf{1}_G\|_r = O(\eta^{(1-1/D)(1/D + 1/q)}).$$

Summing the three estimates gives
$$\|(\varphi^{1/D} - \psi^{1/D})\theta\|_r = O \left(\frac{1}{\eta^{1-1/D}} \|\varphi - \psi\|_p + \eta^{(1-1/D)(1/D + 1/q)} \right).$$
The lemma then follows by letting
$$\|\varphi - \psi\|_p = \eta^{(1 - 1/D)(1/q + 1/D + 1)} \quad \text{and} \quad q = \frac{D(Dr + 1)}{(D-1)(D+1)} > r.$$
Note that $q$ is finite when $D > 1$.
\end{proof}

\begin{proof}[Proof of \cref{cor:robust}]
Let
\begin{align*}
g^\text{add}_\varepsilon(x) &= \frac{1}{\varepsilon^D} g^\text{add}(x/\varepsilon)\\
\tilde f(x) &= [(1 - \delta) (f * g^\text{add}_\varepsilon) + \delta g^\text{out}](x)
\end{align*}
be the densities of $\varepsilon Y^\text{add}$ and $\tilde X$.
Since $X, Y^\text{add}$ and $Y^\text{out}$ have moderate tails, so does $\tilde X$ (with the same $\alpha$ but possibly a different $C$). Similarly, all moments of $\tilde X$ are finite.

Therefore it suffices to 
let $p = r+1$ and 
show
\begin{align*}
\|f - \tilde f\|_\infty &= O(\varepsilon + \delta).
\\
W_p(f, \tilde f) &= O(\varepsilon + \delta^{1/p})
\end{align*}

We first estimate $\|f - \tilde f\|_\infty$.
For each $x \in \mathbb{R}^D$,
\begin{align*}
&\lvert \tilde f(x) - f(x) \rvert
\\&\leq (1-\delta) \int g^\text{add}_\varepsilon(x-y) \lvert f(y) - f(x) \rvert dy + \delta \lvert g^\text{out} - f \rvert
\\&\leq \int \lvert x-y \rvert g^\text{add}_\varepsilon(x-y) \frac{\lvert f(y) - f(x) \rvert}{\lvert x-y \rvert} dy + O(\delta) & \text{($f, g^\text{out}$ bounded)}
\\&= O(\int \lvert x-y \rvert g^\text{add}_\varepsilon(x-y) dy) + O(\delta) & \text{($f$ Lipschitz)}.
\end{align*}
$f$ is Lipschitz because $f^{1/D}$ is and $x \mapsto x^D$ is Lipschitz on $[0, \|f\|_\infty^{1/D}]$.
For the first term,
$$\int \lvert x-y \rvert g^\text{add}_\varepsilon(x-y) dy = \int \lvert y \rvert g^\text{add}_\varepsilon(y) dy = E \lvert \varepsilon Y^\text{add} \rvert = O(\varepsilon).$$
The first estimate then follows.

For $W_p(f, \tilde f)$, let $\Pi^\text{out}$ be a coupling between $f(x)dx$ and $g^\text{out}(y) dy$. Then $$\| \lvert x-y \rvert\|_{L^p(\Pi^\text{out})} = O(1),$$
because all moments of $f$ and $g^\text{out}$ are finite, the two densities have finite Wasserstein distances from the dirac delta measure at $0$.

Consider the following coupling between $P$ and $\tilde P$:
$$d\Pi(x, y) = (1-\delta) g^\text{add}_\varepsilon(y - x) f(x) dxdy + \delta d\Pi^\text{out}(x, y).$$
Then
\begin{align*}
W_p(f, \tilde f)^p
&\leq \|\lvert x - y \rvert\|_{L^p(\Pi)}^p
\\&= \int (1 - \delta) \lvert x - y \rvert^p g_\varepsilon (y - x) f(x) dx dy + \delta \|\lvert x-y \rvert\|_{L^p(\Pi^\text{out})}^p
\\&\leq \int [\lvert z \rvert^p g_\varepsilon (z) f(x) dx dz + O(\delta) & (z = y - x)
\\&= E[\lvert \varepsilon  Y^\text{add} \rvert^{p}] + O(\delta)
\\&= O(\varepsilon^p + \delta).
\end{align*}
The result then follows.

\end{proof}

\subsection{Proofs of Properties in \cref{sec:properties_further}}
\label{sec:proofs_further_properties}

In this subsection, we establish properties of RDAD in \cref{sec:properties_further} in logical order. The results proven in previous sections rely on these properties. Since none of the following proofs depend on those previous results, with the sole exception that \cref{prop:lipschitz_continuity} depends on H\"{o}lder's inequality (\cref{lem:Holder_submeasure}), which does not depend on any results in this section, our arguments are not circular.

\begin{proof}[Proof of \cref{lem:variational_characterization}]
The argument is similar to that in the proof of Proposition 2.2 of \cite{buchet16_DTM_approximation}.

Fix $x$ and fix a mass-$m$ measure $\nu$ that is subordinate to $P$. Consider the random variable $Y = f(X)^{1/D} d(X, x)$.

Let $F_P^{-1}$ and $F_\nu^{-1}$ be the quantile functions of $Y$ under $P$ and $\nu$ respectively. Then $m \cdot RDAD(x)^2 = \int_0^m F_P^{-1}(q)^2 dq$.

The change of variable formula gives an analogous expression for $\nu$:
$$\int [f(\xi)^{1/D} d(\xi, x)]^2 d \nu(\xi) = \int Y^2 dF_\nu = \int_0^m F_\nu^{-1} (t)^2 dt.$$

Since $\nu$ is suborindate to $P$, 
\begin{equation}\label{eqn:lemma_proof_P_nu_quantile_comparison}
F_P^{-1} \leq F_\nu^{-1},
\end{equation}
hence $$\int [f(\xi)^{1/D} d(\xi, x)]^2 d \nu(\xi) \geq \int_0^m F_P^{-1}(q)^2 dq = m \cdot RDAD(x)^2.$$

This proves one inequality in the lemma.

To finish the proof, it suffices to find a $\nu$ such that equality holds on $(0, m)$ in \cref{eqn:lemma_proof_P_nu_quantile_comparison}. 
Let 
\begin{align*}
E_< &= \{\xi: f(\xi)^{1/D} d(\xi, x) < F^{-1}_P(m)\}\\ 
E_= &= \{\xi: f(\xi)^{1/D} d(\xi, x) = F^{-1}_P(m)\}.
\end{align*}
Define 
$$\nu = P \mid E_< + (m - P(E_<))\frac{P \mid E_=}{P(E_=)}.$$
This $\nu$ has the desired property, because by construction, $F_P(y) = F_\nu(y)$ for $y < F^{-1}_P(m)$, and hence $F_P^{-1}(t) = F_\nu^{-1}(t)$ for $t < m$. The result then follows.
\end{proof}

\begin{proof}[Proofs of \cref{prop:lipschitz_continuity,prop:RDAD_approximates_DAD}]

We first prove the Lipschitz continuity of RDAD. Then we prove the convergence of RDAD to DAD. Lipschitz continuity of DAD then follows from the locally uniform convergence of RDAD to DAD.

For Lipschitz continuity of RDAD, when $2 < p < \infty$,
\cref{lem:variational_characterization} implies
$$RDAD(x) = \frac{1}{\sqrt{m}}\|f(\cdot)^{1/D} d(\cdot, x)\|_{L^2(\nu_x)}$$
for some mass-$m$ measure $\nu_x$.
Now, letting $q = (2^{-1} - p^{-1})^{-1}$
\begin{align*}
RDAD(y)
& \leq \frac{1}{\sqrt{m}}\|f(\cdot)^{1/D} d(\cdot, y)\|_{L^2(\nu_x)}
\\& \leq \frac{1}{\sqrt{m}}\|f(\cdot)^{1/D} d(x, y) + f(\cdot)^{1/D} d(\cdot, x)\|_{L^2(\nu_x)}
\\& \leq \frac{1}{\sqrt{m}}\|f(\cdot)^{1/D} d(x, y)\|_{L^2(\nu_x)} + \frac{1}{\sqrt{m}} \|f(\cdot)^{1/D} d(\cdot, x)\|_{L^2(\nu_x)}
\\& = \frac{1}{\sqrt{m}}\|f(\cdot)^{1/D}\|_{L^2(\nu_x)} d(x, y) + RDAD(x)
\\& \leq \frac{1}{\sqrt{m}}m^{1/q}\|f^{1/D}\|_{L^p(P)} d(x, y) + RDAD(x).
\\& = \left[\frac{1}{m^{1/p}} \left(\int f(\xi)^{1+p/D} d\xi\right)^{1/p}\right] d(x, y) + RDAD(x).
\end{align*}
Lipschitz continuity of RDAD then follows by interchanging $x$ and $y$. 

When $f$ is bounded, the above argument follows through if we take $p = \infty$. The following adaptations are necessary: $1/q$ is now $1/2$, and $\|f^{1/D}\|_{L^p(P)}$ in the second last line and the factor in the square bracket in the last line each needs to be replaced by $\|f^{1/D}\|_\infty = \|f\|_\infty^{1/D}$.

For the convergence of RDAD to DAD, since $RDAD(x)$ is the $L^2$ average of $F_x^{-1}(q)$ on $[0, m]$ and $F_x^{-1}$ is increasing,
$$DAD(x) = \lim_{q \to 0^+} F_x^{-1}(q) \leq RDAD(x) \leq F_x^{-1}(m).$$

For each $x$, the right-hand side converges to $\lim_{q \to 0^+} F_x^{-1}(q)$ as $m \to 0$ by definition, and the limit is just $DAD(x)$. Pointwise convergence then follows.

For uniform convergence on compact sets under the assumption that $f$ is bounded, Lipscthiz continuity of RDAD implies $\{RDAD(\cdot; f, m)\}_{m \in (0, 1]}$ is equicontinuous. By Arzela-Ascoli theorem, pointwise convergence of the family implies uniform convergence on compact sets.
\end{proof}

Before proving 
\cref{prop:functional_consistency},
we need a lemma, which is the counterpart of Lemma 8 in \cite{chazal18_DTM_statistics}:

\begin{lemma}
\label{lem:lemma8_modified}
Let $f: \mathbb{R}^D \to \mathbb{R}$ be a density with bounded support and a finite upper bound. Let $P$ be the probability measure on $\mathbb{R}^D$ with density $f$. Let $0 < m < 1$. Then $F_x^{-1}(m)$ is Lipschitz in $x$ on $\mathbb{R}^D$, and $(F_x^{-1}(m))^2$ is Lipschitz in $x$ on every compact set in $\mathbb{R}^D$.
\end{lemma}

\begin{proof}
The second claim follows from the first because the square of any Lipschitz function on a compact metric space is always Lipschitz.

To establish the first claim, we have, by definition
$$F_x^{-1}(m) = \inf \{t: P \left (f(X)^{1/D} d(x, X) \leq t \right )\geq m\}.$$
Since for any $x'$,
$$\{y: f(y)^{1/D} d(x, y) \leq t\} \subseteq \{y: f(y)^{1/D} d(x', y) \leq t + \|f\|_\infty^{1/D} d(x, x')\},$$
whenever $$P \left (f(X)^{1/D} d(x, X) \leq t \right)\geq m$$
holds, we have
$$P \left (f(X)^{1/D} d(x', X) \leq t + \|f\|_\infty^{1/D} d(x, x') \right) \geq m.$$
Taking infimum over $t$ gives
$$F_{x'}^{-1}(m) \leq F_{x}^{-1}(m) + \|f\|_\infty^{1/D} d(x, x').$$
Interchanging $x$ and $x'$ shows $F_{x}^{-1}(m)$ is Lipschitz in $x$ with Lipschitz constant bounded by $\|f\|_\infty^{1/D}$.
\end{proof}

We are now in the position to prove 
\cref{prop:functional_consistency}.
\begin{proof}[Proofs of \cref{prop:functional_consistency}]
We adapt the proofs of Theorems 5 and 9 of \cite{chazal18_DTM_statistics}, the former of which is a stepping stone towards proving Theorem 9 therein.

Under the assumption of uniform continuity, $$\|\hat f_{k_{den}} - f\|_\infty \to 0$$ almost surely (Theorem 4.2 of \cite{biau_devroye_nearest_neigbhor}). Then almost surely,
$$\max_i \lvert C_{N, k_\text{den}, D} 1/d_i - f(X_i)^{1/D} \rvert \to 0,$$
hence
$$\sup_{x \in K} \left \lvert \widehat{RDAD}^2(x) - \frac{1}{k_\text{DTM}} \sum [f(X_{(i)})^{1/D} d(x, X_{(i)})]^2 \right \lvert = o_P(1).$$

The rest of the 
proof of \cref{prop:functional_consistency} 
then follows by replacing each instance of $F_x$ in the proofs of Theorems 5 and 9 of \cite{chazal18_DTM_statistics} with 
\begin{equation}\label{eqn:updated_Fx}
\mathcal{F}_x(t) = P((f(X)^{1/D} d(x, X))^2 \leq t).
\end{equation}
Note that $\mathcal{F}_x(t) = F_x(\sqrt{t})$, where $F_x$ is defined in \cref{eqn:RDAD_quantile}, which is different from the $F_x$ in \cite{chazal18_DTM_statistics}, and this is the origin of the square roots in our statement, which are not present in its counterpart in \cite{chazal18_DTM_statistics}. 

Below, we sketch the main ideas for completeness. We will highlight non-trivial adaptations and refer the readers to \cite{chazal18_DTM_statistics} for details.

The proof starts with the observation
$$\sqrt{n} (\widehat{RDAD}^2(x) - RDAD^2(x)) = A_n(x) + R_n(x),$$
where
\begin{align*}
A_n(x) &= \frac{1}{m} \int_0^{\mathcal{F}_x^{-1}(m)} \sqrt{n}(\mathcal{F}_x(t) - \hat{\mathcal{F}_x}(t)) dt\\
R_n(x) &= \frac{1}{m} \int_{\mathcal{F}_x^{-1}(m)}^{\hat{\mathcal{F}}_x^{-1}(m)} \sqrt{n}(m - \hat{\mathcal{F}_x}(t)) dt,
\end{align*}
where $\hat{\mathcal{F}_x}$ is defined as in \cref{eqn:updated_Fx} with the probability replaced by the empirical measure.

$R_n$ can be bounded with the estimate
$$\abs{R_n(x)} \leq \frac{\sqrt{n}}{m} \abs{S_n(x)}\abs{T_n(x)},$$
where
$$S_n(x)
= \abs{\mathcal{F}_x^{-1}(m) - \hat{\mathcal{F}}_x^{-1}(m)}, \quad T_n(x)
= \sup_t \abs{\mathcal{F}_x(t) - \hat{\mathcal{F}_x}(t)}.
$$

By exactly the same empirical process theory argument as in \cite{chazal18_DTM_statistics} (with all instances of Lemma 8 in \cite{chazal18_DTM_statistics} replaced by \cref{lem:lemma8_modified} here), one can show $$\sup_{x \in K} \lvert S_n(x) \rvert = o_P(1).$$

For $T_n$, by Vapnik-Chernvonenkis theorem,
\begin{align*}
\sup_{x \in K} \lvert T_n(x) \rvert
\leq \sup_{\substack{z \in \mathbb{R}^D\\ r > 0}} 
\lvert \hat{P_n}(E_{z, r}) - P(E_{z, r})\rvert 
= O(1/\sqrt{n}).
\end{align*}
The theorem is indeed applicable because the relevant classifying functions for the $E_{z, r}$'s
form a subset of a finite-dimensional space, specifically:
$$f(\cdot)^{2/D} d(z, \cdot)^2 - t \in \text{span} \{f(\cdot)^{2/D} \|\cdot\|^2, f(\cdot)^{2/D} z_1, ..., f(\cdot)^{2/D} z_D, f(\cdot)^{2/D}, t\}.$$

It remains to show $A_n$ converges to a Gaussian process. Let $\nu_n = \sqrt{n} (\hat P_n - P)$. Note that
$$A_n = \frac{1}{m} \int f_x(y) d\nu_n(y),$$
where
$$f_x(y) = \max(0, \mathcal{F}_x^{-1}(m) - (f(y)\abs{x - y })^2),$$
By Donsker's theorem, it suffices to show $f_x$'s form a $P$-Donsker family. 
By Example 19.7 of \cite{vaart98_asymptoticStat}, indeed they do:
\begin{align*}
\lvert f_x(y) - f_{x'}(y) \rvert
&\leq \lvert \mathcal{F}_x^{-1}(m) - \mathcal{F}_{x'}^{-1}(m)\rvert + \|x - x'\| (\|x\| + \|x'\| + 2\|y\|) \lvert f(y)\rvert,
\\& \leq \lvert \mathcal{F}_x^{-1}(m) - \mathcal{F}_{x'}^{-1}(m) \rvert + \|x - x'\| (2 \text{diam } K + \text{diam } \text{supp} f) \|f\|_\infty,
\end{align*}
where the first term is bounded by a multiple of $\|x - x'\|$ because \cref{lem:lemma8_modified} shows $\mathcal{F}_x^{-1}(m) = (F_x^{-1}(m))^2$ is Lipschitz in $x$.

Finally, the claim on the covariance kernel follows from passing the convariance of $A_n$ to the limit.

\end{proof}

\section{Details of Simulations}
\label{sec:simulation_parameters}
We give details on the constructions of our synthetic datasets, and the parameters used in our experiments, in this section. 
Model
parameters are summarized in \cref{tab:variables_model}. The sample sizes and the density estimation parameters $k_\text{den}$, which depends on the sample sizes, are summarized in \cref{tab:variables_sample_size_k}.

\begin{table}[h]
\begin{tabularx}{\linewidth}{c|c|X}
parameters & values & meaning \\\hline
$k_\text{den}$ & $\lceil (\log_{10} N)^2 \rceil $ & density estimated by the $k_\text{den}$-nearest neighbor estimator; $N$ is the sample size
\\
$m_\text{DTM}$ & 0.002 & amount of mass taken into account by the distance-to-measure setup\\
$\Delta x$ (two-square) & 0.02 & grid size at which the filtration functions are evaluated in the two-square experiments\\
$\Delta x$ (Voronoi) & 0.01 & grid size at which the filtration functions are evaluated in the Voronoi experiment\\
$\Delta x$ (cellular towers) & 0.260 & grid size at which the filtration functions are evaluated in the cellular tower experiment
\\
$B$ & 100 & the number of bootstrap samples\\
$\alpha$ & 0.05 & Confidence sets are bottleneck metric balls whose radii are $(1 - \alpha)$-percentile of the bottleneck distances of the empirical persistence diagram from the diagrams of the bootstrap samples.
\end{tabularx}
\caption{Model parameters used in the simulations.}
\label{tab:variables_model}
\end{table}

\begin{table}[h]
\centering
\begin{tabularx}{.6\linewidth}{c|cc}
datasets          & $N$ & $k_\text{den}$ 
\\\hline
two-square -- David and Goliath & 500 & 8\\
two-square -- all others & 5000 & 14\\
Voronoi -- noisy & 10676 & 17\\
cellular towers & 23389 & 20
\end{tabularx}
\caption{Sample sizes ($N$) and density estimation paremeter $k_\text{den}$ For different datasets.}
\label{tab:variables_sample_size_k}
\end{table}

For two-square datasets, the precise values used are summarized in 
\cref{tab:bisquare_variables_setup}.
The inner radius $r$ and outer radius $R$ of a square annulus refer to half of the sidelengths of the inner and outer squares in a square annulus. This is illustrated in \cref{fig:TDA_illustration_circular_annulus}.

\begin{table}
\begin{tabularx}{\linewidth}{c|c|c|X}
parameters & David \& Goliath & other datasets & meanings \\\hline
$(x_1, x_2)$ & (0, 4) & (0, 4) & $x$-coordinates of centers of the square annuli\\
$(y_1, y_2)$ & (0, 0) & (0, 0) & $y$-coordinates of centers of the square annuli \\
$(p_1, p_2)$ & (0.4, 0.6) & (0.5, 0.5) & masses of the square annuli \\
$(r_1, r_2)$ & (1, 0.1) & (1, 1/3) & inner radii of the square annuli, see \cref{fig:TDA_illustration_circular_annulus}\\
$(R_1, R_2)$ & (1.1, 0.12) & (1.4, 1.4/3) & outer radii of the square annuli, see \cref{fig:TDA_illustration_circular_annulus}\\
$(\sigma_1, \sigma_2)$ & -- & (0.15, 0.05) & standard deviations of the isotropic Gaussian noises on the square annuli\\
$N_\text{outliers}$ & -- & 8 & number of outliers
\end{tabularx}
\caption{Parameters used to generate two-square datasets (in alphabetical order, Greek letters follow Latin ones).
}
\label{tab:bisquare_variables_setup}
\end{table}

\begin{figure}[h]
\centering
\includegraphics[width=0.35\linewidth]{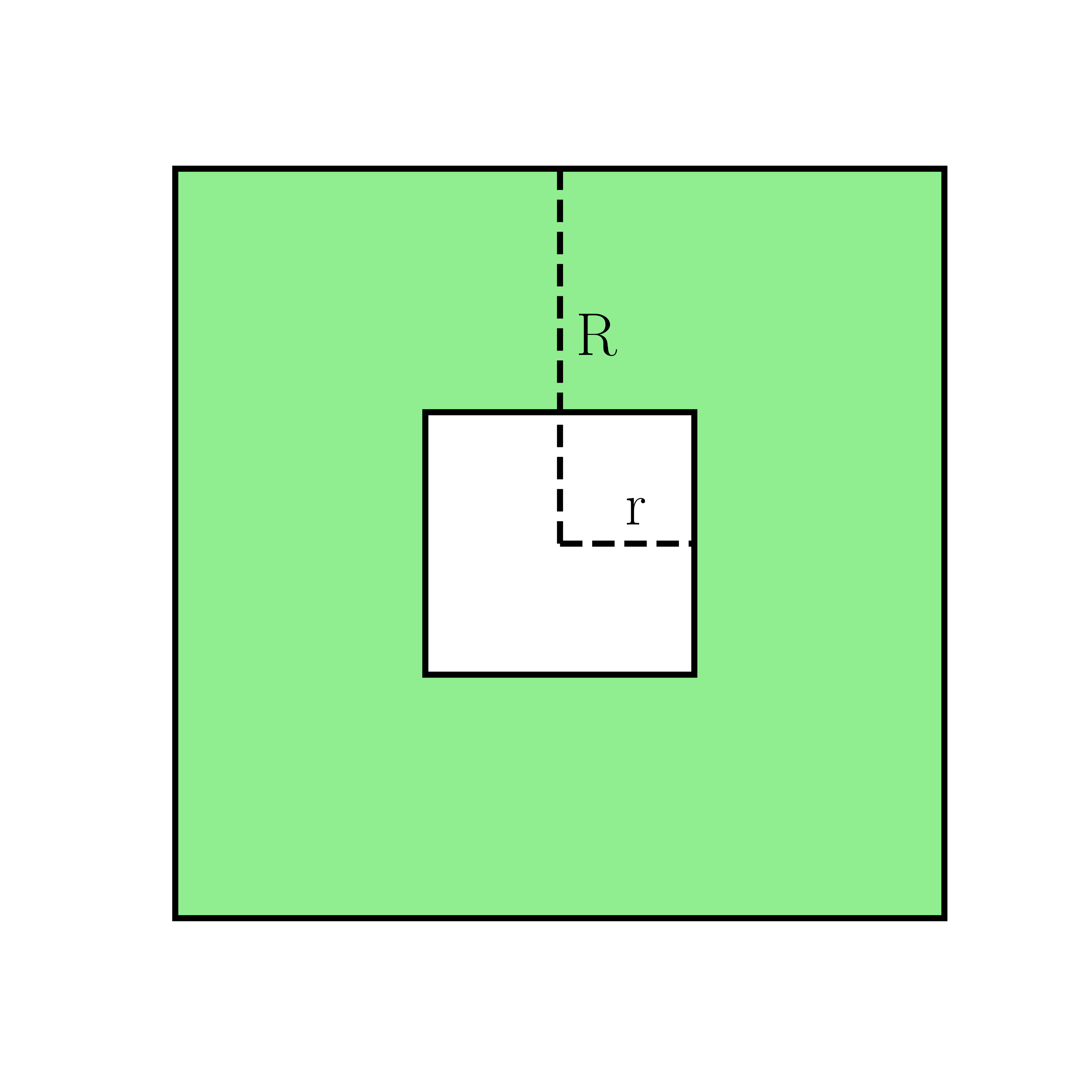}
\caption{Inner and outer radii of a square annulus}
\label{fig:TDA_illustration_circular_annulus}
\end{figure}

For the Voronoi dataset, we first describe the data generation process. The actual values of parameters used is summarized in \cref{tab:voronoi_variables_setup}.

\begin{description}
\item[Generation of the Voronoi diagram] A fixed number $M^+_\text{cell}$ of centers of Voronoi cells are sample points on an infinite strip $\mathbb{R} \times [-y_+, y_+]$ so that more points will be near the central vertical line $y = 0$. Specifically, the $x$- and $y$-coordinates are sampled independently from a biexponential distribution $\text{Biexp}(\lambda)$ with scale parameter $\lambda$ and a uniform distribution respectively. The Voronoi diagram is then generated from these Voronoi cell centers. Since more points are sampled at the central vertical line, cells near the line are smaller than those on the right. 

\item[Generation of a super-sample] 
We sample $N^+$ points from the equal-weight mixture of the uniform measures supported on the boundaries of the cells. Then smaller cells have edges with higher edge density.  

\item[Corruption by additive noise] To add noise to the dataset, for each sample point $(x, y)$ on edges of the Voronoi diagram, we perturb it with independent mean-zero isotropic Gaussian noise, whose standard deviation is $\sigma_0 e^{\lvert x \rvert/\lambda}$,
so points near the central vertical line are corrupted by a smaller noise.

\item[Removal of ill-behaving points to obtain the sample] Since cells near the boundary, even if finite, tend to be very elongated, we discard all points outside of a rectangle $R$. This motivates our choice that outliers lie in $R$, because outliers lying outside of $R$ will be discarded. We analyze the dataset formed by the remaining $N$ points. The framing rectangles of the scatter plots as well as plots of significant loops in \cref{sec:voronoi} are all the rectangle $R$.

\end{description}

\begin{table}
\begin{tabularx}{\linewidth}{c|c|X}
parameters & values & meaning\\\hline
$M_\text{cell}$    &    88 & number of Voronoi cells completely contained in the rectangle $R$\\
$M^+_\text{cell}$  &   200 & number of cells in the full Voronoi diagram\\
$N^+$              & 20000 & number of sample points in the full Voronoi diagram\\
$p_\text{outlier}$ & 0.002 & proportion of the $N^+$ sample points that are replaced by outliers\\
$R$                &    -- & the rectangle $[-x_0, x_0] \times [-y_0, y_0]$ only on which sample points are passed to topological computation\\
$x_0$              &     3 & maximum absolute value of the $x$-coordinates of sample points that are passed to topological computation\\
$y_0$              &     1 & maximum absolute value of the $y$-coordinates of sample points that are passed to topological computation\\
$y_+$              &     2 & maximum possible absolute values of the $y$-coordinates of Voronoi cell centers \\
$\lambda$          &     1 & scale parameter of the biexponential distribution, from which the $x$-coordinates of Voronoi cell centers are drawn\\
$\sigma_0$         &  0.01 & scale parameter of the additive Gaussian noise
\end{tabularx}
\caption{Parameters used to generate Voronoi datasets}
\label{tab:voronoi_variables_setup}
\end{table}

For the cellular tower dataset, it is preprocessed as follows. Only towers in the contiguous United States are retained. Incorrectly labelled towers are left as-is, except that one Texas tower, which is erroneously labelled to be in the middle of the Atlantic Ocean, is removed. We treat the longitude and latitude of each of the remaining 23389 towers as the $x$- and $y$-coordinates of a data point.

From the data points, the filtration function values are evaluated on the grid, on the rectangle $[-126, -65.8] \times [23.9, 50.0]$, which contains all points. The grid size is $1/100^{th}$ of the shorter side of the rectangle.

\end{appendices}

\bibliography{../paper/tex/bib}


\begin{thebibliography}{45}
\ifx \bisbn   \undefined \def \bisbn  #1{ISBN #1}\fi
\ifx \binits  \undefined \def \binits#1{#1}\fi
\ifx \bauthor  \undefined \def \bauthor#1{#1}\fi
\ifx \batitle  \undefined \def \batitle#1{#1}\fi
\ifx \bjtitle  \undefined \def \bjtitle#1{#1}\fi
\ifx \bvolume  \undefined \def \bvolume#1{\textbf{#1}}\fi
\ifx \byear  \undefined \def \byear#1{#1}\fi
\ifx \bissue  \undefined \def \bissue#1{#1}\fi
\ifx \bfpage  \undefined \def \bfpage#1{#1}\fi
\ifx \blpage  \undefined \def \blpage #1{#1}\fi
\ifx \burl  \undefined \def \burl#1{\textsf{#1}}\fi
\ifx \doiurl  \undefined \def \doiurl#1{\url{https://doi.org/#1}}\fi
\ifx \betal  \undefined \def \betal{\textit{et al.}}\fi
\ifx \binstitute  \undefined \def \binstitute#1{#1}\fi
\ifx \binstitutionaled  \undefined \def \binstitutionaled#1{#1}\fi
\ifx \bctitle  \undefined \def \bctitle#1{#1}\fi
\ifx \beditor  \undefined \def \beditor#1{#1}\fi
\ifx \bpublisher  \undefined \def \bpublisher#1{#1}\fi
\ifx \bbtitle  \undefined \def \bbtitle#1{#1}\fi
\ifx \bedition  \undefined \def \bedition#1{#1}\fi
\ifx \bseriesno  \undefined \def \bseriesno#1{#1}\fi
\ifx \blocation  \undefined \def \blocation#1{#1}\fi
\ifx \bsertitle  \undefined \def \bsertitle#1{#1}\fi
\ifx \bsnm \undefined \def \bsnm#1{#1}\fi
\ifx \bsuffix \undefined \def \bsuffix#1{#1}\fi
\ifx \bparticle \undefined \def \bparticle#1{#1}\fi
\ifx \barticle \undefined \def \barticle#1{#1}\fi
\bibcommenthead
\ifx \bconfdate \undefined \def \bconfdate #1{#1}\fi
\ifx \botherref \undefined \def \botherref #1{#1}\fi
\ifx \url \undefined \def \url#1{\textsf{#1}}\fi
\ifx \bchapter \undefined \def \bchapter#1{#1}\fi
\ifx \bbook \undefined \def \bbook#1{#1}\fi
\ifx \bcomment \undefined \def \bcomment#1{#1}\fi
\ifx \oauthor \undefined \def \oauthor#1{#1}\fi
\ifx \citeauthoryear \undefined \def \citeauthoryear#1{#1}\fi
\ifx \endbibitem  \undefined \def \endbibitem {}\fi
\ifx \bconflocation  \undefined \def \bconflocation#1{#1}\fi
\ifx \arxivurl  \undefined \def \arxivurl#1{\textsf{#1}}\fi
\csname PreBibitemsHook\endcsname

\bibitem{icke91_cosmic_void_voronoi}
\begin{barticle}
\bauthor{\bsnm{{Icke}}, \binits{V.}},
\bauthor{\bsnm{{van de Weygaert}}, \binits{R.}}:
\batitle{{The galaxy distribution as a {Voronoi} foam}}.
\bjtitle{QJRAS}
\bvolume{32},
\bfpage{85}--\blpage{112}
(\byear{1991})
\end{barticle}
\endbibitem

\bibitem{carlsson09_topodata}
\begin{barticle}
\bauthor{\bsnm{Carlsson}, \binits{G.}}:
\batitle{Topology and data}.
\bjtitle{Bulletin of the American Mathematical Society}
\bvolume{46},
\bfpage{255}--\blpage{308}
(\byear{2009})
\end{barticle}
\endbibitem

\bibitem{chazal21_TDA_survey_data}
\begin{botherref}
\oauthor{\bsnm{Chazal}, \binits{F.}},
\oauthor{\bsnm{Michel}, \binits{B.}}:
An introduction to topological data analysis: Fundamental and practical aspects
  for data scientists.
Frontiers in Artificial Intelligence
\textbf{4}
(2021).
\doiurl{10.3389/frai.2021.667963}
\end{botherref}
\endbibitem

\bibitem{perea19_TDA_survey_time_series}
\begin{barticle}
\bauthor{\bsnm{Perea}, \binits{J.A.}}:
\batitle{Topological time series analysis}.
\bjtitle{AMS Notices}
\bvolume{66}(\bissue{5}),
\bfpage{686}--\blpage{694}
(\byear{2019})
\end{barticle}
\endbibitem

\bibitem{aktas19_TDA_survey_network}
\begin{barticle}
\bauthor{\bsnm{Aktas}, \binits{M.E.}},
\bauthor{\bsnm{Akbas}, \binits{E.}},
\bauthor{\bsnm{Fatmaoui}, \binits{A.E.}}:
\batitle{Persistence homology of networks: methods and applications}.
\bjtitle{Applied Network Science}
\bvolume{4}(\bissue{1}),
\bfpage{61}
(\byear{2019}).
\doiurl{10.1007/s41109-019-0179-3}
\end{barticle}
\endbibitem

\bibitem{buchet18_TDA_survey_material_science}
\begin{bbook}
\bauthor{\bsnm{Buchet}, \binits{M.}},
\bauthor{\bsnm{Hiraoka}, \binits{Y.}},
\bauthor{\bsnm{Obayashi}, \binits{I.}}:
In: \beditor{\bsnm{Tanaka}, \binits{I.}} (ed.)
\bbtitle{Persistent Homology and Materials Informatics},
pp. \bfpage{75}--\blpage{95}.
\bpublisher{Springer},
\blocation{Singapore}
(\byear{2018}).
\doiurl{10.1007/978-981-10-7617-6_5}.
\burl{https://doi.org/10.1007/978-981-10-7617-6_5}
\end{bbook}
\endbibitem

\bibitem{xu19_cosmic_void_TDA}
\begin{barticle}
\bauthor{\bsnm{Xu}, \binits{X.}},
\bauthor{\bsnm{Cisewski-Kehe}, \binits{J.}},
\bauthor{\bsnm{Green}, \binits{S.B.}},
\bauthor{\bsnm{Nagai}, \binits{D.}}:
\batitle{Finding cosmic voids and filament loops using topological data
  analysis}.
\bjtitle{Astronomy and Computing}
\bvolume{27},
\bfpage{34}--\blpage{52}
(\byear{2019}).
\doiurl{10.1016/j.ascom.2019.02.003}
\end{barticle}
\endbibitem

\bibitem{salch21_TDA_survey_MRI}
\begin{barticle}
\bauthor{\bsnm{Salch}, \binits{A.}},
\bauthor{\bsnm{Regalski}, \binits{A.}},
\bauthor{\bsnm{Abdallah}, \binits{H.}},
\bauthor{\bsnm{Suryadevara}, \binits{R.}},
\bauthor{\bsnm{Catanzaro}, \binits{M.J.}},
\bauthor{\bsnm{Diwadkar}, \binits{V.A.}}:
\batitle{From mathematics to medicine: A practical primer on topological data
  analysis (tda) and the development of related analytic tools for the
  functional discovery of latent structure in fmri data}.
\bjtitle{PLOS ONE}
\bvolume{16}(\bissue{8}),
\bfpage{1}--\blpage{33}
(\byear{2021}).
\doiurl{10.1371/journal.pone.0255859}
\end{barticle}
\endbibitem

\bibitem{stolz17_smallFeatures_networks}
\begin{barticle}
\bauthor{\bsnm{Stolz}, \binits{B.J.}},
\bauthor{\bsnm{Harrington}, \binits{H.A.}},
\bauthor{\bsnm{Porter}, \binits{M.A.}}:
\batitle{Persistent homology of time-dependent functional networks constructed
  from coupled time series}.
\bjtitle{Chaos: An Interdisciplinary Journal of Nonlinear Science}
\bvolume{27}(\bissue{4}),
\bfpage{047410}
(\byear{2017}).
\doiurl{10.1063/1.4978997}
\end{barticle}
\endbibitem

\bibitem{feng21_smallFeatures_voting}
\begin{barticle}
\bauthor{\bsnm{Feng}, \binits{M.}},
\bauthor{\bsnm{Porter}, \binits{M.A.}}:
\batitle{Persistent homology of geospatial data: A case study with voting}.
\bjtitle{SIAM Review}
\bvolume{63}(\bissue{1}),
\bfpage{67}--\blpage{99}
(\byear{2021}).
\doiurl{10.1137/19M1241519}
\end{barticle}
\endbibitem

\bibitem{jaquette20_smallFeatures_fractalDimension}
\begin{barticle}
\bauthor{\bsnm{Jaquette}, \binits{J.}},
\bauthor{\bsnm{Schweinhart}, \binits{B.}}:
\batitle{Fractal dimension estimation with persistent homology: A comparative
  study}.
\bjtitle{Communications in Nonlinear Science and Numerical Simulation}
\bvolume{84},
\bfpage{105163}
(\byear{2020}).
\doiurl{10.1016/j.cnsns.2019.105163}
\end{barticle}
\endbibitem

\bibitem{bubenik20_smallFeatures_curvature}
\begin{barticle}
\bauthor{\bsnm{Bubenik}, \binits{P.}},
\bauthor{\bsnm{Hull}, \binits{M.}},
\bauthor{\bsnm{Patel}, \binits{D.}},
\bauthor{\bsnm{Whittle}, \binits{B.}}:
\batitle{Persistent homology detects curvature}.
\bjtitle{Inverse Problems}
\bvolume{36}(\bissue{2}),
\bfpage{025008}
(\byear{2020}).
\doiurl{10.1088/1361-6420/ab4ac0}
\end{barticle}
\endbibitem

\bibitem{motta18_smallFeatures_hexagonalLattices}
\begin{barticle}
\bauthor{\bsnm{Motta}, \binits{F.C.}},
\bauthor{\bsnm{Neville}, \binits{R.}},
\bauthor{\bsnm{Shipman}, \binits{P.D.}},
\bauthor{\bsnm{Pearson}, \binits{D.A.}},
\bauthor{\bsnm{Bradley}, \binits{R.M.}}:
\batitle{Measures of order for nearly hexagonal lattices}.
\bjtitle{Physica D: Nonlinear Phenomena}
\bvolume{380-381},
\bfpage{17}--\blpage{30}
(\byear{2018}).
\doiurl{10.1016/j.physd.2018.05.005}
\end{barticle}
\endbibitem

\bibitem{xia14_smallFeatures_protein}
\begin{barticle}
\bauthor{\bsnm{Xia}, \binits{K.}},
\bauthor{\bsnm{Wei}, \binits{G.-W.}}:
\batitle{Persistent homology analysis of protein structure, flexibility, and
  folding}.
\bjtitle{International Journal for Numerical Methods in Biomedical Engineering}
\bvolume{30}(\bissue{8}),
\bfpage{814}--\blpage{844}
(\byear{2014}).
\doiurl{10.1002/cnm.2655}
\end{barticle}
\endbibitem

\bibitem{aragoncalvo12_cosmic_void_hierarchy}
\begin{barticle}
\bauthor{\bsnm{Aragon-Calvo}, \binits{M.A.}},
\bauthor{\bsnm{Szalay}, \binits{A.S.}}:
\batitle{{The hierarchical structure and dynamics of voids}}.
\bjtitle{Monthly Notices of the Royal Astronomical Society}
\bvolume{428}(\bissue{4}),
\bfpage{3409}--\blpage{3424}
(\byear{2012}).
\doiurl{10.1093/mnras/sts281}
\end{barticle}
\endbibitem

\bibitem{wilding21_cosmic_void_TDA}
\begin{barticle}
\bauthor{\bsnm{Wilding}, \binits{G.}},
\bauthor{\bsnm{Nevenzeel}, \binits{K.}},
\bauthor{\bparticle{van~de} \bsnm{Weygaert}, \binits{R.}},
\bauthor{\bsnm{Vegter}, \binits{G.}},
\bauthor{\bsnm{Pranav}, \binits{P.}},
\bauthor{\bsnm{Jones}, \binits{B.J.T.}},
\bauthor{\bsnm{Efstathiou}, \binits{K.}},
\bauthor{\bsnm{Feldbrugge}, \binits{J.}}:
\batitle{{Persistent homology of the cosmic web -- I. Hierarchical topology in
  $\Lambda$CDM cosmologies}}.
\bjtitle{Monthly Notices of the Royal Astronomical Society}
\bvolume{507}(\bissue{2}),
\bfpage{2968}--\blpage{2990}
(\byear{2021}).
\doiurl{10.1093/mnras/stab2326}
\end{barticle}
\endbibitem

\bibitem{bendich16_brainArteryTrees}
\begin{barticle}
\bauthor{\bsnm{Bendich}, \binits{P.}},
\bauthor{\bsnm{Marron}, \binits{J.S.}},
\bauthor{\bsnm{Miller}, \binits{E.}},
\bauthor{\bsnm{Pieloch}, \binits{A.}},
\bauthor{\bsnm{Skwerer}, \binits{S.}}:
\batitle{{Persistent homology analysis of brain artery trees}}.
\bjtitle{The Annals of Applied Statistics}
\bvolume{10}(\bissue{1}),
\bfpage{198}--\blpage{218}
(\byear{2016}).
\doiurl{10.1214/15-AOAS886}
\end{barticle}
\endbibitem

\bibitem{carlsson09_multipersistence}
\begin{botherref}
\oauthor{\bsnm{Carlsson}, \binits{G.}},
\oauthor{\bsnm{Zomorodian}, \binits{A.}}:
The theory of multidimensional persistence.
Discrete Comput. Geom.,
71--93
(2009)
\end{botherref}
\endbibitem

\bibitem{sheehy12_multicover}
\begin{bchapter}
\bauthor{\bsnm{Sheehy}, \binits{D.R.}}:
\bctitle{A multicover nerve for geometric inference}.
In: \bbtitle{CCCG: Canadian Conference in Computational Geometry}
(\byear{2012})
\end{bchapter}
\endbibitem

\bibitem{blumberg21_bipersistence_stability}
\begin{botherref}
\oauthor{\bsnm{Blumberg}, \binits{A.J.}},
\oauthor{\bsnm{Lesnick}, \binits{M.}}:
Stability of 2-Parameter Persistent Homology
(2021)
\end{botherref}
\endbibitem

\bibitem{lesnick15_interactive_bipersistence}
\begin{botherref}
\oauthor{\bsnm{Lesnick}, \binits{M.}},
\oauthor{\bsnm{Wright}, \binits{M.}}:
Interactive Visualization of {2-D} Persistence Modules
(2015)
\end{botherref}
\endbibitem

\bibitem{moon18_persistence_terrace}
\begin{barticle}
\bauthor{\bsnm{Moon}, \binits{C.}},
\bauthor{\bsnm{Giansiracusa}, \binits{N.}},
\bauthor{\bsnm{Lazar}, \binits{N.A.}}:
\batitle{Persistence terrace for topological inference of point cloud data}.
\bjtitle{Journal of Computational and Graphical Statistics}
\bvolume{27}(\bissue{3}),
\bfpage{576}--\blpage{586}
(\byear{2018}).
\doiurl{10.1080/10618600.2017.1422432}
\end{barticle}
\endbibitem

\bibitem{berry19_kNN_graph}
\begin{barticle}
\bauthor{\bsnm{Berry}, \binits{T.}},
\bauthor{\bsnm{Sauer}, \binits{T.}}:
\batitle{Consistent manifold representation for topological data analysis}.
\bjtitle{Foundations of Data Science}
\bvolume{1}(\bissue{1}),
\bfpage{1}--\blpage{38}
(\byear{2019})
\end{barticle}
\endbibitem

\bibitem{hickok22_density_scaled_filtration}
\begin{botherref}
\oauthor{\bsnm{Hickok}, \binits{A.}}:
A Family of Density-Scaled Filtered Complexes
(2022)
\end{botherref}
\endbibitem

\bibitem{bell19_weighted_persistence}
\begin{barticle}
\bauthor{\bsnm{Bell}, \binits{G.}},
\bauthor{\bsnm{Lawson}, \binits{A.}},
\bauthor{\bsnm{Martin}, \binits{J.}},
\bauthor{\bsnm{Rudzinski}, \binits{J.}},
\bauthor{\bsnm{Smyth}, \binits{C.}}:
\batitle{Weighted persistent homology}.
\bjtitle{Involve}
\bvolume{12}(\bissue{5}),
\bfpage{823}--\blpage{837}
(\byear{2019}).
\doiurl{10.2140/involve.2019.12.823}
\end{barticle}
\endbibitem

\bibitem{chazal11_DTM}
\begin{barticle}
\bauthor{\bsnm{Chazal}, \binits{F.}},
\bauthor{\bsnm{Cohen-Steiner}, \binits{D.}},
\bauthor{\bsnm{M{\'e}rigot}, \binits{Q.}}:
\batitle{Geometric inference for probability measures}.
\bjtitle{Found Comput Math}
\bvolume{11},
\bfpage{733}--\blpage{751}
(\byear{2011}).
\doiurl{10.1007/s10208-011-9098-0}
\end{barticle}
\endbibitem

\bibitem{buchet16_DTM_approximation}
\begin{barticle}
\bauthor{\bsnm{Buchet}, \binits{M.}},
\bauthor{\bsnm{Chazal}, \binits{F.}},
\bauthor{\bsnm{Oudot}, \binits{S.Y.}},
\bauthor{\bsnm{Sheehy}, \binits{D.R.}}:
\batitle{Efficient and robust persistent homology for measures}.
\bjtitle{Computational Geometry}
\bvolume{58},
\bfpage{70}--\blpage{96}
(\byear{2016}).
\doiurl{10.1016/j.comgeo.2016.07.001}
\end{barticle}
\endbibitem

\bibitem{chazal18_DTM_statistics}
\begin{barticle}
\bauthor{\bsnm{Chazal}, \binits{F.}},
\bauthor{\bsnm{Fasy}, \binits{B.}},
\bauthor{\bsnm{Lecci}, \binits{F.}},
\bauthor{\bsnm{Michel}, \binits{B.}},
\bauthor{\bsnm{Rinaldo}, \binits{A.}},
\bauthor{\bsnm{Wasserman}, \binits{L.}}:
\batitle{Robust topological inference: Distance to a measure and kernel
  distance}.
\bjtitle{Journal of Machine Learning Research}
\bvolume{18},
\bfpage{1}--\blpage{40}
(\byear{2018})
\end{barticle}
\endbibitem

\bibitem{anai19_DTM_generalizations}
\begin{bchapter}
\bauthor{\bsnm{Anai}, \binits{H.}},
\bauthor{\bsnm{Chazal}, \binits{F.}},
\bauthor{\bsnm{Glisse}, \binits{M.}},
\bauthor{\bsnm{Ike}, \binits{Y.}},
\bauthor{\bsnm{Inakoshi}, \binits{H.}},
\bauthor{\bsnm{Tinarrage}, \binits{R.}},
\bauthor{\bsnm{Umeda}, \binits{Y.}}:
\bctitle{{DTM-Based Filtrations}}.
In: \beditor{\bsnm{Barequet}, \binits{G.}},
\beditor{\bsnm{Wang}, \binits{Y.}} (eds.)
\bbtitle{35th International Symposium on Computational Geometry (SoCG 2019)}.
\bsertitle{Leibniz International Proceedings in Informatics (LIPIcs)},
vol. \bseriesno{129},
pp. \bfpage{58}--\blpage{15815}.
\bpublisher{Schloss Dagstuhl--Leibniz-Zentrum fuer Informatik},
\blocation{Dagstuhl, Germany}
(\byear{2019}).
\doiurl{10.4230/LIPIcs.SoCG.2019.58}.
\burl{http://drops.dagstuhl.de/opus/volltexte/2019/10462}
\end{bchapter}
\endbibitem

\bibitem{edelsbrunner10comptopo}
\begin{bbook}
\bauthor{\bsnm{Edelsbrunner}, \binits{H.}},
\bauthor{\bsnm{Harer}, \binits{J.}}:
\bbtitle{Computational Topology: An Introduction}.
\bsertitle{Applied Mathematics}.
\bpublisher{American Mathematical Society},
\blocation{Providence, RI}
(\byear{2010}).
\burl{https://www.ams.org/books/mbk/069/}
\end{bbook}
\endbibitem

\bibitem{otter17_persistent_homology}
\begin{barticle}
\bauthor{\bsnm{Otter}, \binits{N.}},
\bauthor{\bsnm{Porter}, \binits{M.A.}},
\bauthor{\bsnm{Tillmann}, \binits{U.}},
\bauthor{\bsnm{Grindrod}, \binits{P.}},
\bauthor{\bsnm{Harrington}, \binits{H.A.}}:
\batitle{A roadmap for the computation of persistent homology}.
\bjtitle{EPJ Data Science}
\bvolume{6}(\bissue{1}),
\bfpage{17}
(\byear{2017}).
\doiurl{10.1140/epjds/s13688-017-0109-5}
\end{barticle}
\endbibitem

\bibitem{fasy14_confidence_set}
\begin{barticle}
\bauthor{\bsnm{Fasy}, \binits{B.T.}},
\bauthor{\bsnm{Lecci}, \binits{F.}},
\bauthor{\bsnm{Rinaldo}, \binits{A.}},
\bauthor{\bsnm{Wasserman}, \binits{L.}},
\bauthor{\bsnm{Balakrishnan}, \binits{S.}},
\bauthor{\bsnm{Singh}, \binits{A.}}:
\batitle{Confidence sets for persistence diagrams}.
\bjtitle{The Annals of Statistics}
\bvolume{42}(\bissue{6}),
\bfpage{2301}--\blpage{2339}
(\byear{2014})
\end{barticle}
\endbibitem

\bibitem{silverman_density}
\begin{bbook}
\bauthor{\bsnm{Silverman}, \binits{B.W.}}:
\bbtitle{Density Estimation for Statistics and Data Analysis},
\bedition{1st} edn.
\bpublisher{Chapman \& Hall/ CRC},
\blocation{Boca Raton, FL}
(\byear{2003})
\end{bbook}
\endbibitem

\bibitem{biau_devroye_nearest_neigbhor}
\begin{bbook}
\bauthor{\bsnm{Biau}, \binits{G.}},
\bauthor{\bsnm{Devroye}, \binits{L.}}:
\bbtitle{Lectures on the Nearest Neighbor Method},
\bedition{1st} edn.
\bpublisher{Springer},
\blocation{Cham, Switzerland}
(\byear{2015}).
\burl{https://link.springer.com/book/10.1007/978-3-319-25388-6}
\end{bbook}
\endbibitem

\bibitem{gudhi_CubicalComplex}
\begin{bchapter}
\bauthor{\bsnm{Dlotko}, \binits{P.}}:
\bctitle{Cubical complex}.
In: \bbtitle{{GUDHI} User and Reference Manual},
\bedition{{3.4.1}} edn.
\bpublisher{{GUDHI Editorial Board}},
\blocation{Saclay, France}
(\byear{2021}).
\burl{https://gudhi.inria.fr/doc/3.4.1/group__cubical__complex.html}
\end{bchapter}
\endbibitem

\bibitem{pranav16_cosmic_void_TDA}
\begin{barticle}
\bauthor{\bsnm{Pranav}, \binits{P.}},
\bauthor{\bsnm{Edelsbrunner}, \binits{H.}},
\bauthor{\bparticle{van~de} \bsnm{Weygaert}, \binits{R.}},
\bauthor{\bsnm{Vegter}, \binits{G.}},
\bauthor{\bsnm{Kerber}, \binits{M.}},
\bauthor{\bsnm{Jones}, \binits{B.J.T.}},
\bauthor{\bsnm{Wintraecken}, \binits{M.}}:
\batitle{{The topology of the cosmic web in terms of persistent Betti
  numbers}}.
\bjtitle{Monthly Notices of the Royal Astronomical Society}
\bvolume{465}(\bissue{4}),
\bfpage{4281}--\blpage{4310}
(\byear{2016}).
\doiurl{10.1093/mnras/stw2862}
\end{barticle}
\endbibitem

\bibitem{HIFLD21_cellurlar_towers}
\begin{botherref}
\oauthor{\bsnm{HIFLD}}:
Cellular Towers.
\url{https://hifld-geoplatform.opendata.arcgis.com/datasets/cellular-towers/explore?location=26.819085%2C-53.792376%2C2.76&showTable=true}
\end{botherref}
\endbibitem

\bibitem{scikit-learn}
\begin{barticle}
\bauthor{\bsnm{Pedregosa}, \binits{F.}},
\bauthor{\bsnm{Varoquaux}, \binits{G.}},
\bauthor{\bsnm{Gramfort}, \binits{A.}},
\bauthor{\bsnm{Michel}, \binits{V.}},
\bauthor{\bsnm{Thirion}, \binits{B.}},
\bauthor{\bsnm{Grisel}, \binits{O.}},
\bauthor{\bsnm{Blondel}, \binits{M.}},
\bauthor{\bsnm{Prettenhofer}, \binits{P.}},
\bauthor{\bsnm{Weiss}, \binits{R.}},
\bauthor{\bsnm{Dubourg}, \binits{V.}},
\bauthor{\bsnm{Vanderplas}, \binits{J.}},
\bauthor{\bsnm{Passos}, \binits{A.}},
\bauthor{\bsnm{Cournapeau}, \binits{D.}},
\bauthor{\bsnm{Brucher}, \binits{M.}},
\bauthor{\bsnm{Perrot}, \binits{M.}},
\bauthor{\bsnm{Duchesnay}, \binits{E.}}:
\batitle{Scikit-learn: Machine learning in {P}ython}.
\bjtitle{Journal of Machine Learning Research}
\bvolume{12},
\bfpage{2825}--\blpage{2830}
(\byear{2011})
\end{barticle}
\endbibitem

\bibitem{numpy}
\begin{barticle}
\bauthor{\bsnm{Harris}, \binits{C.R.}},
\bauthor{\bsnm{Millman}, \binits{K.J.}},
\bauthor{\bparticle{van~der} \bsnm{Walt}, \binits{S.J.}},
\bauthor{\bsnm{Gommers}, \binits{R.}},
\bauthor{\bsnm{Virtanen}, \binits{P.}},
\bauthor{\bsnm{Cournapeau}, \binits{D.}},
\bauthor{\bsnm{Wieser}, \binits{E.}},
\bauthor{\bsnm{Taylor}, \binits{J.}},
\bauthor{\bsnm{Berg}, \binits{S.}},
\bauthor{\bsnm{Smith}, \binits{N.J.}},
\bauthor{\bsnm{Kern}, \binits{R.}},
\bauthor{\bsnm{Picus}, \binits{M.}},
\bauthor{\bsnm{Hoyer}, \binits{S.}},
\bauthor{\bparticle{van} \bsnm{Kerkwijk}, \binits{M.H.}},
\bauthor{\bsnm{Brett}, \binits{M.}},
\bauthor{\bsnm{Haldane}, \binits{A.}},
\bauthor{\bparticle{del} \bsnm{R{\'{i}}o}, \binits{J.F.}},
\bauthor{\bsnm{Wiebe}, \binits{M.}},
\bauthor{\bsnm{Peterson}, \binits{P.}},
\bauthor{\bsnm{G{\'{e}}rard-Marchant}, \binits{P.}},
\bauthor{\bsnm{Sheppard}, \binits{K.}},
\bauthor{\bsnm{Reddy}, \binits{T.}},
\bauthor{\bsnm{Weckesser}, \binits{W.}},
\bauthor{\bsnm{Abbasi}, \binits{H.}},
\bauthor{\bsnm{Gohlke}, \binits{C.}},
\bauthor{\bsnm{Oliphant}, \binits{T.E.}}:
\batitle{Array programming with {NumPy}}.
\bjtitle{Nature}
\bvolume{585}(\bissue{7825}),
\bfpage{357}--\blpage{362}
(\byear{2020}).
\doiurl{10.1038/s41586-020-2649-2}
\end{barticle}
\endbibitem

\bibitem{scipy}
\begin{barticle}
\bauthor{\bsnm{Virtanen}, \binits{P.}},
\bauthor{\bsnm{Gommers}, \binits{R.}},
\bauthor{\bsnm{Oliphant}, \binits{T.E.}},
\bauthor{\bsnm{Haberland}, \binits{M.}},
\bauthor{\bsnm{Reddy}, \binits{T.}},
\bauthor{\bsnm{Cournapeau}, \binits{D.}},
\bauthor{\bsnm{Burovski}, \binits{E.}},
\bauthor{\bsnm{Peterson}, \binits{P.}},
\bauthor{\bsnm{Weckesser}, \binits{W.}},
\bauthor{\bsnm{Bright}, \binits{J.}},
\bauthor{\bsnm{{van der Walt}}, \binits{S.J.}},
\bauthor{\bsnm{Brett}, \binits{M.}},
\bauthor{\bsnm{Wilson}, \binits{J.}},
\bauthor{\bsnm{Millman}, \binits{K.J.}},
\bauthor{\bsnm{Mayorov}, \binits{N.}},
\bauthor{\bsnm{Nelson}, \binits{A.R.J.}},
\bauthor{\bsnm{Jones}, \binits{E.}},
\bauthor{\bsnm{Kern}, \binits{R.}},
\bauthor{\bsnm{Larson}, \binits{E.}},
\bauthor{\bsnm{Carey}, \binits{C.J.}},
\bauthor{\bsnm{Polat}, \binits{{\. I}.}},
\bauthor{\bsnm{Feng}, \binits{Y.}},
\bauthor{\bsnm{Moore}, \binits{E.W.}},
\bauthor{\bsnm{{VanderPlas}}, \binits{J.}},
\bauthor{\bsnm{Laxalde}, \binits{D.}},
\bauthor{\bsnm{Perktold}, \binits{J.}},
\bauthor{\bsnm{Cimrman}, \binits{R.}},
\bauthor{\bsnm{Henriksen}, \binits{I.}},
\bauthor{\bsnm{Quintero}, \binits{E.A.}},
\bauthor{\bsnm{Harris}, \binits{C.R.}},
\bauthor{\bsnm{Archibald}, \binits{A.M.}},
\bauthor{\bsnm{Ribeiro}, \binits{A.H.}},
\bauthor{\bsnm{Pedregosa}, \binits{F.}},
\bauthor{\bsnm{{van Mulbregt}}, \binits{P.}},
\bauthor{\bsnm{{SciPy 1.0 Contributors}}}:
\batitle{{{SciPy} 1.0: Fundamental Algorithms for Scientific Computing in
  Python}}.
\bjtitle{Nature Methods}
\bvolume{17},
\bfpage{261}--\blpage{272}
(\byear{2020}).
\doiurl{10.1038/s41592-019-0686-2}
\end{barticle}
\endbibitem

\bibitem{pandas_software}
\begin{botherref}
\oauthor{\bparticle{pandas~development} \bsnm{team}, \binits{T.}}:
Pandas-dev/pandas: Pandas.
\doiurl{10.5281/zenodo.3509134}.
\url{https://doi.org/10.5281/zenodo.3509134}
\end{botherref}
\endbibitem

\bibitem{pandas_article}
\begin{bchapter}
\bauthor{\bparticle{{W}es} \bsnm{{M}c{K}inney}}:
\bctitle{{D}ata {S}tructures for {S}tatistical {C}omputing in {P}ython}.
In: \beditor{\bparticle{{S}t\'efan van~der} \bsnm{{W}alt}},
\beditor{\bparticle{{J}arrod} \bsnm{{M}illman}} (eds.)
\bbtitle{{P}roceedings of the 9th {P}ython in {S}cience {C}onference},
pp. \bfpage{56}--\blpage{61}
(\byear{2010}).
\doiurl{10.25080/Majora-92bf1922-00a}
\end{bchapter}
\endbibitem

\bibitem{lam15_numba}
\begin{bchapter}
\bauthor{\bsnm{Lam}, \binits{S.K.}},
\bauthor{\bsnm{Pitrou}, \binits{A.}},
\bauthor{\bsnm{Seibert}, \binits{S.}}:
\bctitle{Numba: A {LLVM}-based python {JIT} compiler}.
In: \bbtitle{Proceedings of the Second Workshop on the LLVM Compiler
  Infrastructure in HPC},
pp. \bfpage{1}--\blpage{6}
(\byear{2015})
\end{bchapter}
\endbibitem

\bibitem{matplotlib}
\begin{barticle}
\bauthor{\bsnm{Hunter}, \binits{J.D.}}:
\batitle{Matplotlib: A {2D} graphics environment}.
\bjtitle{Computing in science \& engineering}
\bvolume{9}(\bissue{3}),
\bfpage{90}--\blpage{95}
(\byear{2007})
\end{barticle}
\endbibitem

\bibitem{vaart98_asymptoticStat}
\begin{bbook}
\bauthor{\bsnm{Vaart}, \binits{A.W.v.d.}}:
\bbtitle{Asymptotic Statistics}.
\bsertitle{Cambridge Series in Statistical and Probabilistic Mathematics}.
\bpublisher{Cambridge University Press},
\blocation{the United Kingdom}
(\byear{1998}).
\doiurl{10.1017/CBO9780511802256}
\end{bbook}
\endbibitem

\end{thebibliography}

\end{document}